\documentclass[11pt,a4paper]{amsart}
\usepackage{amssymb}
\usepackage[dvips]{color}
\usepackage{amsmath}
\usepackage{amsthm}
\usepackage{enumitem}
\usepackage{amsfonts}
\usepackage{bbm}
\usepackage{xcolor}
\usepackage{pst-node}
\usepackage{tikz-cd}

\allowdisplaybreaks[4] 
\newtheoremstyle{myremark}     {10pt}{10pt}{}{}{\bfseries}{.}{.5em}{}

 \newtheorem{thm}{Theorem}[section]
 
 \newtheorem{lem}[thm]{Lemma}
 \newtheorem{prop}[thm]{Proposition}
 \theoremstyle{definition}
 \newtheorem{defn}[thm]{Definition}

  \theoremstyle{myremark}
  \newtheorem{rem}[thm]{Remark}

 \newcommand{\h}{\mathcal{H}}

 \newcommand{\J}{\mathcal{J}}

 \newcommand{\F}{\mathcal{F}}
 \newcommand{\B}{\mathbf{B}}
 
 \newcommand{\N}{\mathbb{N}}

 \newcommand{\R}{\mathbb{R}}
 \newcommand{\C}{\mathbb{C}}

 \newcommand{\abs}[1]{\left\vert#1\right\vert}
 \newcommand{\set}[1]{\left\{#1\right\}}
 
 \newcommand{\norm}[1]{\left\Vert#1\right\Vert}
 
 \newcommand\inner[2]{\left\langle #1, #2 \right\rangle}

\begin{document}
\title[Complete Metric Approximation Property]{Complete Metric Approximation Property For Mixed $q$-Deformed Araki-Woods Factors }

\author[Bikram]{Panchugopal Bikram}
\author[Mohanta]{Rajeeb Mohanta}
\author[Mukherjee]{Kunal Krishna Mukherjee}

\address{School of Mathematical Sciences, National Institute of Science Education and Research, Bhubaneswar, A CI of HBNI, Jatani- 752050, India}
\email{bikram@niser.ac.in, rajeeb.mohanta@niser.ac.in}
\address{Department of Mathematics, IIT Madras, Chennai - 600036, India.}
\email{kunal@iitm.ac.in}


\begin{abstract}
The main result of this paper is to establish the weak* completely metric approximation property (w*-CMAP) for the mixed $q$-deformed Araki-Woods factors for all symmetric matrices with entries $-1< q_{i,j}< 1$, using an ultraproduct embedding of the mixed $q$-deformed Araki-Woods factors.
\end{abstract}

\maketitle

\section{Introduction}
Among all approximation properties of $C^*$-algebras and von Neumann algebras, the strongest approximation property is the (weak*) completely positive approximation property (w*-CPAP). In the case of $C^*$-algebras, Effros and Choi showed that the completely positive approximation property is equivalent to nuclearity in ~\cite{EC78}. Connes, in his work on the classification of injective factors ~\cite{Co76}, showed that the weak* completely positive approximation property of von Neumann algebras is equivalent to injectivity. For non-injective von Neumann algebras, there are two approximation properties that arise as weak forms of the completely positive approximation property, namely the Haagerup property and the weak* completely (bounded) metric approximation property(w*-CMAP). In the group context, the amenability of a discrete group $G$ corresponds to the existence of an approximate identity in the Fourier algebra $A(G)$ consisting of finitely supported normalized positive definite functions. When the finite support assumption is relaxed, and the identity in $A(G)$ can be approximated by normalized positive definite functions that vanish at infinity, one gets the Haagerup property. When one can approximate the identity by some finitely supported functions, which are uniformly bounded in the completely bounded Fourier multiplier norm, the group is said to be weakly amenable. The (weak*) complete metric approximation property is the generalization of weak amenability in the operator algebraic setup. The study of finite approximation properties of operator algebras has provided many landmark results.

 For finite non-injective von Neumann algebras, in ~\cite{OP10} Ozawa and Popa showed that the w*-CMAP is intimately connected to several remarkable indecomposability results, such as strong solidity, solidity, absence of Cartan subalgebras, primeness, and so on. One of the important properties of von Neumann algebras studied recently is strong solidity, introduced by Ozawa and Popa ~\cite{OP10}. In the same paper, Ozawa and Popa showed that free-group factors are strongly solid. Among other finite non-injective von Neumann algebras, both q-Gaussian von Neumann algebras and mixed $q$-Gaussian von Neumann algebras have w*-CMAP and are strongly solid, shown in ~\cite{A11} and ~\cite{JZ15} respectively. In the literature, some finite von Neumann algebras coming from commutation relations given by Yang-Baxter operators are also studied. Some results related to the factoriality of such von Neumann algebras can be found in \cite{K06} and \cite{BKMMS21}. However, not much information on the approximation properties and indecomposability results of these von Neumann algebras is available so far.
 
 Recently, several advancements have been made regarding these indecomposability results in the case of type $\mathrm{III}$ von Neumann algebras. The study of indecomposability results of type $\mathrm{III}$ von Neumann algebras is known to be technically difficult and involved. However, it is fundamental to understand their structure and properties since they arise naturally and are still not completely understood.  In ~\cite{HR11}, Houdayer and Ricard showed that the free Araki-Woods algebras have the w*-CMAP, and these algebras do not have Cartan subalgebras. The $q$-Araki-Woods algebras were shown to have the w*-CMAP by  Avsec, Brannan, and Wasilewski in ~\cite{ABW18}. In ~\cite{BKM20},  the mixed $q$-deformed Araki-Woods von Neumann algebra is constructed out of operators satisfying canonical commutation relations deformed by a symmetric matrix. The mixed $q$-deformed Araki-Woods von Neumann algebras share many properties with the $q$-Araki-Woods von Neumann algebras. The same article ~\cite{BKM20} discusses many properties of the mixed $q$-deformed Araki-Woods algebras, such as factoriality, Haagerup property, etc. In ~\cite{BKM21}, the mixed $q$-deformed Araki-Woods von Neumann algebras are shown to be non-injective in most of the cases. It was in 2023 that the question of factoriality and non-injectivity of the mixed $q$-deformed Araki Woods von Neumann algebras was settled in full generality in two separate instances as in \cite{KSW23} and \cite{BKM23}. Henceforth, we refer to the mixed $q$-deformed Araki Woods von Neumann algebras as mixed $q$-deformed Araki Woods von Neumann factors. Hence, it is natural to ask whether these von Neumann algebras have the w*-CMAP. This article discusses the w*-CMAP of mixed $q$-deformed Araki-Woods von Neumann factors. The following is the main result of this article.
\begin{thm}
The mixed $q$-deformed Araki-Woods factor $\Gamma_T(\h_{\R},U_t)^{\prime\prime}$ has the w*-complete metric approximation property.
\end{thm}

We follow the approach of the free case by Houdayer and Ricard, which characterizes a natural class of completely bounded maps, called radial multipliers and uses their completely bounded norms. In order to estimate the completely bounded norm of the radial multipliers, we use the techniques of \cite{ABW18}.

This article is organized as follows. In Section 2, we recall some basics of operator space, the ultraproduct of von Neumann algebras, and the construction of the mixed $q$-deformed Araki-Woods von Neumann factors. In Section 3, we establish the ultraproduct embedding of the mixed $q$-deformed Araki-Woods von Neumann factors and use it to estimate the cb norm of the radial multipliers on the mixed $q$-deformed Araki-Woods von Neumann factors. Finally, in section 4, we provide the proof of our main theorem.

\section{Preliminary }

\subsection{Operator Spaces}
We briefly recall some notions of operator spaces that are crucial to the context of this article. An operator space is a  Banach space $X$ together with a sequence of norm $\set{\norm{\cdot}_n}_{n\geqslant 1}$ which satisfies Ruan's axiom. In the category of operator spaces, the important class of maps is the completely bounded maps. A linear map $\phi$ between two operator spaces $X$ and $Y$ is called completely bounded if it's completely bounded norm (cb norm), defined by 
\begin{align*}
\norm{\phi}_{cb}:= \sup_{n\in\N}\norm{\phi^{(n)}:M_n(X)\rightarrow M_n(Y)}
\end{align*}
is finite, where $\phi^{(n)}$ is the $n$-th amplification of $\phi$ defined by  $\phi^{(n)}((x_{i,j}))= (\phi(x_{i,j}))$ for $(x_{i,j})\in M_n(X)$ and $n\in\N$. An operator space is called homogeneous if every bounded linear map $\phi: X\rightarrow X$ is completely bounded and $\norm{\phi}_{cb} = \norm{\phi}$.

We use the column and row operator space structure of an abstract Hilbert space $H$ given by:
\begin{align*}
    H_c = \B(\C, H) \ \text{and} \ H_r = \B(\overline{H}, \C) 
\end{align*}
respectively, where $\overline{H}$ denotes the complex conjugate of the Hilbert space $H$. It follows from \cite[Section 3.4]{ER22} that the column  Hilbert space $H_c$ as well as the row Hilbert space $(H_r)$ are homogeneous. 
Let $X$ and $Y$ be two operator spaces. Let $x = (x_{ij}) \in M_n(X\otimes Y)$, $n\in\N.$ Define
\begin{align*}
    \norm{x}_{h,n}=\inf_{r\geqslant 1}\set{\norm{y_{ij}}_{M_{n,r}(X)}\norm{z_{ij}}_{M_{r,n}(Y)}: x_{ij}=\sum_{k=1}^r y_{ik}\otimes z_{kj}}
\end{align*}
The operator space tensor product defined by these norms is called the Haagerup tensor product of $X$ and $Y$ and is denoted as $X\otimes_{h}Y$.
\begin{rem}\label{isometric}
    Let $H$ and $K$ be two Hilbert space. For row and column Hilbert spaces we have the following classical isomorphisms ( \cite[Prop. 9.3.1, Prop 9.3.4]{ER22}):
\begin{enumerate}
    \item $H_c \otimes_{min} K_r \simeq H_c\otimes_{h} K_r$
    \item $H_c \otimes_h \overline{K}_r\simeq \mathcal{K}(K,H),$ where  $\mathcal{K}(K,H)$ denotes the set of all compact operators from $K$ to $H$. The complete isometry is given by the map $(\xi\otimes_h\eta)(k) = \inner{\eta}{k}\xi$, for $\xi\in H$ and $\eta, k\in K$.
\end{enumerate}

\end{rem}
We recall the following definition of   approximation properties of operator spaces that are central to this article. 
\begin{defn}

\begin{enumerate}
	\item 
An operator space $X$ is said to have the completely bounded approximation property (CBAP) if there exists a net $(\phi_i)_{i\in I}$ of finite rank completely bounded maps on $X$ such that\\
\begin{enumerate}
	\item 
$\sup_{i\in I}\norm{\phi_i}_{cb}<\infty$, and \\

\item $\underset{i}{\lim } \norm{\phi_i(x)-x}=0$ for every $x\in X$. \\
\end{enumerate} 
Additionally,  if  $\sup_{i\in I}\norm{\phi_i}_{cb}<1,$ then $X$ is said to have the complete metric approximation property (CMAP).\\

\item 
In case of the operator space $X$ is a dual of some operator space, then $X$ is said to have the $w^*$-complete metric approximation property ($w^*$-CMAP) if there exists a net $(\phi_i)_{i\in I}$ of finite rank $w^*$-continuous completely bounded maps on $X$ such that \\
\begin{enumerate}
	\item 

$\norm{\phi_i}_{cb}\leqslant 1$ for each $i\in I$, and\\
\item  $\underset{i}{\lim } ~\phi_i(x) =x$ in weak* topology for all $x\in X.$  
\end{enumerate} 
\end{enumerate}
\end{defn}

\subsection{Mixed $q$-deformed Araki-Woods von Neumann factors}
We recall the mixed $q$-deformed Araki-Woods von Neumann factors constructed in \cite{BKM20}. Let $\h_{\R}$ be a separable real Hilbert space and $\R\ni t\mapsto U_{t} $ be a strongly continuous orthogonal representation of $\R$ on $\h_{\R}$. We denote the complexification of $\h_{\R}$ by $\h_{\C}=\h_{\R}\otimes \C$.  Denote the inner product and norm on $\h_{\C}$ by $\inner{\cdot}{\cdot}_{\h_{\C}}$ and $\norm{\cdot}_{\h_{\C}}$ respectively. All inner products in this article are considered to be linear in the second variable.  Identify $\h_{\R}$ in $\h_{\C}$ by $\h_{\R}\otimes 1$. Since $\h_{\C}= \h_{\R}+i \h_{\R}$, as a real Hilbert space the inner product of $\h_{\R}$ in $\h_{\C}$ is given by $\mathfrak{R}\inner{\cdot}{\cdot}_{\h_\C}$.  We denote the complex conjugation on $\h_{\C}$ by $\J$, which is a bounded anti-linear operator.

We can extend the representation $\R\ni t \mapsto U_{t}$ from $ \h_{\R}$ to a strongly continuous one-parameter group of unitaries on $\h_{\C}$. Denote the extension again by $U_t$ for each $t$ with a slight abuse of notation. Let $A$ be the analytic generator of the strongly continuous one-parameter group $\lbrace U_{t}:t\in \R\rbrace$ acting on $\h_{\C}$, which is non-degenerate, positive, and self-adjoint.

With the help of the generator $A$, a new inner product $\langle \cdot,\cdot\rangle_{U}$ can be defined on $\h_{\C}$   as follows:
\begin{align*}
\inner{\xi}{\eta}_{U}= \inner{\frac{2}{1+A^{-1}}\xi}{\eta}_{\h_{\C}}, \text{ for }\xi,\eta \in \h_{\C}.
\end{align*}
Let $\h$ be the completion of $\h_{\C}$ with respect to $\inner{\cdot}{\cdot}_{U}$. We denote the inner product and norm on $\h$ by $\inner{\cdot}{\cdot}_U$ and $\norm{\cdot}_{U}$ respectively. Since $A$ is affiliated to $vN\{U_{t}:t \in \R\}$, we have
\begin{align*}
\inner{U_{t}\xi}{U_{t}\eta}_{U}=\inner{\xi}{\eta}_{U},\text{ for }\xi,\eta\in \h_{\C}.
\end{align*}
Hence, $(U_{t})_{t\in\R}$ extends to a strongly continuous unitary representation $(\widetilde{U_{t}})_{t\in\R}$ of $\R$ on $\h$. Let $\widetilde{A}$ be the analytic generator associated to $(\widetilde{U_{t}})_{t\in\R}$, which is clearly an extension of $A$. From \cite[Prop. 2.1]{BM17} and the discussion prior to it, it follows that the spectral data of $A$ and $\widetilde{A}$ are essentially the same. Therefore, we denote the extensions $(\widetilde{U_{t}})_{t\in\R}$ and $\widetilde{A}$ again by $(U_{t})_{t\in\R}$ and $A$ respectively with slight abuse of notation.

Let $N$ denote $\{1,2,\dots,r\}$, $r\in\N$, or  $\N$. We fix a decomposition of $\h_{\R}$ as follows:
\begin{align*}
\h_{\R}:=\bigoplus\limits_{i\in N}\h_{\R}^{(i)}, 
\end{align*}
where $\h_{\R}^{(i)}$, $i\in N$, are non-trivial invariant subspaces of $\lbrace U_{t}: t\in \R\rbrace$. Fix some real numbers $q_{ij}$ such that, $-1<q_{ij}=q_{ji}<1$ for $i,j\in N$ with $\sup\limits_{ i,j \in N} \abs{q_{ij}}<1 $. In this paper, we will often denote the scalar $q_{ij}$ also by $q(i,j)$ for $i,j\in N$.
Note that, $\h_{\C}=\bigoplus\limits_{i\in N}\h_{\C}^{(i)}$, where $\h_{\C}^{(i)}$ is the complexification of $\h_{\R}^{(i)}$ for $i\in N$. Also, since $\h_{\C}^{(i)}$, $i\in N$, are invariant for $\lbrace U_{t}: t\in \R\rbrace$, it follows that, $\h=\bigoplus\limits_{i\in N}\h^{(i)}$, where $\h^{(i)}$, $i\in N$, are  the completions of $\h_{\C}^{(i)}$, $i\in N$, with respect to $\inner{\cdot}{\cdot}_{U}$. \\
For fix $i,j\in N$, we define $T_{i,j}:\h_{\R}^{(i)}\otimes \h_{\R}^{(j)}\rightarrow \h_{\R}^{(j)}\otimes \h_{\R}^{(i)}$ to be the bounded extension of:
\begin{align*}
 \xi\otimes \eta\mapsto q_{ij} (\eta \otimes \xi), \text{ for }\xi\in \h_{\R}^{(i)},\eta \in \h_{\R}^{(j)}.
\end{align*}
Then, $T_\R:=\oplus_{i,j\in N} T_{i,j}\in \B(\h_\R\otimes\h_\R)$. We extend $T_\R$ to $\h_{\C}\otimes \h_{\C}$ linearly and denote the extension by $T_\C$.
By a simple density argument, it follows that $T_\C$ admits a unique bounded extension $T$ to $\h\otimes \h$. It is easy to verify that $T:=\oplus_{i,j\in N}T_{ij}$, where $T_{ij}:\h^{(i)}\otimes\h^{(j)}\rightarrow\h^{(j)}\otimes\h^{(i)}$ is defined as the extension of the map:
\begin{align}\label{T action}
\xi\otimes \eta\mapsto q_{ij} (\eta \otimes \xi), \text{ for }\xi\in \h^{(i)},\eta \in \h^{(j)}.
\end{align}
Moreover, $T$ has the following properties:
\begin{align}\label{Eq.T}
&T^{*}=T,\quad\quad\quad\quad\quad(\text{since } q_{ij}=q_{ji}, \text{ for }i,j\in N); \\
\nonumber&\norm T_{\h\otimes\h}< 1,\,\,\,\quad\quad (\text{since}\sup_{i,j\in   N} \abs{q_{ij}}<1);\\
\nonumber&(1\otimes T)(T\otimes 1)(1\otimes T)=(T\otimes1)(1\otimes T)(T\otimes 1),
\end{align}
where $1\otimes T \text{ and }T\otimes 1$ are the natural amplifications of $T$ to $\h\otimes \h\otimes\h$. The third relation listed in Eq. \eqref{Eq.T} is referred to as the Yang-Baxter equation.
Let $\F(\h):=\C\Omega\oplus\bigoplus\limits_{n=1}^{\infty}\h^{\otimes n}$ be the full Fock space of $\h$, where $\Omega$ is a distinguished unit vector $($vacuum vector$)$ in $\C$. By convention, $\h^{\otimes 0}:=\C\Omega$. The canonical inner product and norm on $\F(\h)$ will be denoted by $\inner{\cdot}{\cdot}_{\F(\h)}$ and $\norm{\cdot}_{\F(\h)}$ respectively. 
For $\xi\in \h$, let $a(\xi)$ and $a^{*}(\xi)$ denote the canonical left creation and annihilation operators acting on $\F(\h)$ which are defined as follows:
\begin{align}\label{gen op1}
&a(\xi)\Omega =\xi,\\
\nonumber& a(\xi)(\xi_{1}\otimes \xi_{2} \otimes\cdots \otimes \xi_{n})=\xi \otimes\xi_{1}\otimes\xi_{2}\otimes\cdots \otimes \xi_{n}\\ \text{ and},\ \ \
\nonumber& a^{*}(\xi)\Omega=0,\\
\nonumber&a^{*}(\xi)(\xi_{1}\otimes \xi_{2} \otimes\cdots \otimes \xi_{n})=\langle \xi ,\xi_{1}\rangle_{U} \xi_{2}\otimes\cdots \otimes \xi_{n},
\end{align}
where $\xi_{1}\otimes\cdots\otimes\xi_{n}\in\h^{\odot n}$ $(\h^{\odot n}$ denotes the $n$-fold algebraic tensor product of $\h)$ for $n\geq 1$. The operators $a(\xi)$ and $a^{*}(\xi)$ are bounded and adjoints of each other on $\F(\h)$.
Let $T_{i}$ be the operator acting on $\h^{\otimes(i+1)}$ for $i \in \N$ defined as follows:
\begin{align}\label{ExtndTDefn}
T_{i}:=\underset{i-1}{\underbrace{1\otimes\cdots\otimes 1}} \otimes T.
\end{align}
Extend $T_{i}$ to $\h^{\otimes n}$ for all $n>i+1$ by $T_{i}\otimes\underset{n-i-1}{\underbrace{1\otimes\cdots\otimes1}}$ and denote the extension again by $T_i$ with a slight abuse of notation.
Let $S_{n}$ denote the symmetric group of $n$ elements. Note that $S_1$ is trivial. For $n\geq 2$, let $\tau_{i}$ be the transposition between $i$ and $i+1$. It is  known that the set $\lbrace\tau_{i}\rbrace_{i=1}^{n-1}$  generates  $S_{n}$. 
For $n\in \N$, let $\pi:S_{n}\rightarrow\B(\h^{\otimes n})$ be the quasi-multiplicative map given by $\pi(1)=1$ and $\pi(\tau_{i})= T_{i}\text{ } (i=1,\dots,n-1)$. Consider $P^{(n)}\in\B(\h^{\otimes n})$, defined as follows:
\begin{align}\label{Twisted inner product}
P^{(n)}:=\underset{\sigma\in S_{n}}\sum \pi(\sigma).
\end{align}
By convention, $P^{(0)}$ on $\h^{\otimes0}$ is identity.
From the properties of $T$ in Eq. \eqref{Eq.T} and \cite[Theorem.  2.3]{BSp94}, it follows that $P^{(n)}$ is a strictly positive operator for every $n\in \N$. Following \cite{BSp94}, the association 
\begin{align}\label{New Fock inner pdt}
\inner{\xi}{\eta}_{T}=\delta_{n,m}\inner{\xi}{P^{(n)}\eta}_{\F(\h)},\text{ for }\xi\in \h^{\otimes m},\eta\in \h^{\otimes n},
\end{align}
defines a definite sesquilinear form on $\F(\h)$, and, let $\F_{T}(\h)$ denote the completion of $\F(\h)$ with respect to the norm on $\F(\h)$ induced by $\inner{\cdot}{\cdot}_T$. We denote the inner product and the norm on $\F_T(\h)$ by $\inner{\cdot}{\cdot}_T$ and $\norm{\cdot}_T$ respectively. We also denote $\F^{\text{finite}}_{T}(\h):=\text{ span }_\C \{\h^{\otimes n}, n\geq0\rbrace$ and $\h^{\otimes_T^n} =\overline{\h^{\otimes n}}^{\norm{\cdot}_T}$ for $n\in \N$.\\
Note that $\overline{\h}^{\norm{\cdot}_T}=\h$. Following \cite{BSp94}, for $\xi\in \h$, consider the $T$-deformed left creation and annihilation operators on $\mathcal{F}_T(\h)$ defined as follows: 
\begin{align}\label{gen op2}
l(\xi)&:=a(\xi),\text{ and, }\\
\nonumber l^{*}(\xi)&:=\begin{cases}a^{*}(\xi)(1+T_{1}+T_{1}T_{2}+\cdots +T_{1}T_{2}\cdots T_{n-1}), \text{ on }\h^{\otimes n};\\
0, \quad\text{ on }\C\Omega.\end{cases}
\end{align}
Then, $l(\xi)$ and $l^{*}(\xi)$ admit bounded extensions to $\mathcal{F}_T(\h)$ and we have:

\begin{prop}\label{bounded operators}
Let $\xi\in\h$. Then, the following hold:
\begin{enumerate}[label=(\roman*)]
\item If $\norm{T}_{\h\otimes\h} = q < 1$, then $\norm{l(\xi)}\leqslant\norm{\xi}_{U}(1-q)^{-\frac{1}{2}}$.
\item $l(\xi)$ and $l^{*}(\xi)$ are adjoints of each other on $\mathcal{F}_{T}(\h)$.
\end{enumerate}
\end{prop}

\begin{proof}
The proof follows exactly along the same lines of  \cite[Theorem. 3.1]{BSp94}. We omit the details.
\end{proof}

Following Eq. \eqref{gen op2}, the definition of $l^{*}(\xi)$ involves the operator $T$. It follows from Eq. \eqref{T action} that, the action of $l^{*}(\xi)$ on $\mathcal{F}_{T}(\h)$ is  as follows. Fix $n\in \N$, and, $\xi_{i_{k}}\in\h^{(i_{k})}$ for $i_{k}\in N$ and $1\leqslant k\leqslant n$. Then,
\begin{align}\label{left annihilation}
&l^{*}(\xi)\Omega=0, \text{ and, }\\ 
\nonumber&l^{*}(\xi)(\xi_{i_{1}}\otimes\cdots\otimes \xi_{i_{n}})=\underset{k=1}{ \overset{n}\sum}\langle \xi,\xi_{i_{k}}\rangle_{U}q_{i_{k}i_{k-1}}\cdots q_{i_{k}i_{1}}(\xi_{i_{1}}\otimes \xi_{i_{2}}\otimes\cdots\\
\nonumber&\quad\quad\quad\quad\quad\quad\quad\quad\quad\quad\quad\quad\quad\quad\quad\cdots \otimes \xi_{i_{k-1}}\otimes
\xi_{i_{k+1}}\otimes\cdots\otimes \xi_{i_{n}}).
\end{align}

Define 
\begin{align}\label{genop}
s(\xi):= l(\xi)+l^*(\xi), \text{ for }\xi \in \h_{\R}.
\end{align}

We denote the $C^*$-algebra generated by the self-adjoint operators $\{s(\xi):\xi\in\h_\R\}$ as $\Gamma_{T}(\h_{\R},U_{t})$ and the associated von Neumann algebra as $\Gamma_{T}(\h_{\R},U_{t})^{\prime\prime}\subseteq \mathbf{B}(\mathcal{F}_T(\h))$.  The von Neumann algebra  $\Gamma_{T}(\h_{\R},U_{t})^{\prime\prime}$ is called  as  the mixed $q$-deformed Araki-Woods von Neumann factor. This von Neumann algebra is equipped with a vacuum state $\varphi$ given by $\varphi(\cdot)=\inner{\Omega}{\cdot\Omega}_T$.

 For any simple tensor $\xi_{1}\otimes \cdots \otimes\xi_{n}\in \h_{\C}^{\otimes n}$, there exists a unique operator $W(\xi_{1}\otimes \cdots \otimes\xi_{n})$ in $\Gamma_T(\h_{\R},U_t)^{\prime\prime}$ such that $W(\xi_{1}\otimes \cdots \otimes\xi_{n})\Omega = \xi_{1}\otimes \cdots \otimes\xi_{n}$; these operators are called Wick words. In \cite{BKM20}, the following explicit formula for the Wick product is given, which will be used frequently in the remaining parts of the paper.

\begin{lem}\label{wick pdt formula}
Fix $n\in \N$, and let $\xi_{t_{k}}\in\h_\C^{(t_{k})}$ for $1\leqslant t_{k}\leqslant N$, and $1\leqslant k\leqslant n$,
\begin{align*}
 W&(\xi_{t_1}\otimes \cdots \otimes\xi_{t_n})
\\&=\sum_{\substack{ l,m\geq0 \\ l+m=n}}\sum_{\substack{I=\lbrace i(1),\dots ,i(l) \rbrace,\, i(1)< \cdots < i(l)\\ J=\lbrace j(1)\cdots j(m)\rbrace,\,j(1)< \cdots < j(m)\\ I\cup J=\lbrace 1,\cdots ,n\rbrace\\ I\cap J=\emptyset}}\!\!\!\!\!\!\!\! f_{(I,J)}(q_{ij})l(\xi_{t_{i(1)}})\cdots l(\xi_{t_{i(l)}})l^*(\J\xi_{t_{j(1)}})\cdots l^*(\J\xi_{t_{j(m)}}),
\end{align*}
where $f_{(I,J)}(q_{ij})=\underset{\lbrace (r,s):\text{ }1\leqslant r\leqslant l,1\leqslant s\leqslant m, i(r)>j(s)\rbrace}\prod  q_{t_{i(r)},t_{j(s)}}$.
\end{lem}

Now, we fix some notations. Fix $n\in\N$.  For $0\leqslant k\leqslant n $, let $\Pi_{n,k}:=\set{J:=(j(1),\ldots ,j(k)): 1\leqslant j(1)< \cdots <j(k)\leqslant n}$. For each  $J\in\Pi_{n,k}$,  denote the complement of $J$ in $\set{1,\ldots,n}$ by $J^c:=(i(1),\ldots,i(n-k)).$ where $ 1\leqslant i(1)< \cdots <i(n-k)\leqslant n$. 
%, let $\h^n$ be the set of all n-tuples of vectors in $\h$. For $\textbf{f}=(f_1,\ldots,f_n)\in\h^n$ we denote $\textbf{f}^{\otimes n}=f_1\otimes\cdots\otimes f_n\in\h^{\otimes n} . 
Let $\eta_{t_k}\in\h^{(t_k)}$ for $t_k\in N,1\leqslant k\leqslant n.$
Now, for $\eta=(\eta_{t_1},\ldots,\eta_{t_n})\in\h^{(t_1)}\times\cdots\times\h^{(t_n)}$ and $J\in\Pi_{n,k}$, we set 
\begin{align*}
&l(\eta):=l(\eta_{t_1})\cdots l(\eta_{t_n}),\ \ \ l^*(\eta):=l^*(\eta_{t_1})\cdots l^*(\eta_{t_n}),\\
&\eta_{t_J}:=(\eta_{t_{j(1)}},\ldots,\eta_{t_{j(k)}})\in\h^{(t_{j(1)})}\times\cdots\times\h^{(t_{j(k)})},\\
&\eta_{t_{J^c}}:=(\eta_{t_{i(1)}},\ldots,\eta_{t_{i(n-k)}})\in\h^{(t_{i(1)})}\times\cdots\times\h^{(t_{i(n-k)})} , 
\end{align*}
hence,
\begin{align*}
l(\eta_{t_J})=l(\eta_{t_{j(1)}})\cdots l(\eta_{t_{j(k)}}),\ \ \ l^*(\eta_{t_{J^c}})=l^*(\eta_{t_{i(1)}})\cdots l^*(\eta_{t_{i(n-k)}}) .
\end{align*}
We write $\eta_{t_J}^{\otimes k} := \eta_{t_{j(1)}}\otimes\cdots\otimes\eta_{t_{j(k)}}$ and $\eta_{t_{J^c}}^{\otimes (n-k)}:=\eta_{t_{i(1)}}\otimes\cdots\otimes\eta_{t_{i(n-k)}}$.
For $n\in\N$, we denote $\J^n\eta:=\J\eta_{t_1}\otimes\J\eta_{t_2}\otimes\cdots\otimes\J\eta_{t_n}.$

Using these notations, we can now rewrite the wick product formula \ref{wick pdt formula} as follows:
\begin{align}\label{wick pdt2}
W(\eta_{t_1}\otimes \cdots \otimes\eta_{t_n}) =\displaystyle\sum^n_{k=0}\sum_{J\in\Pi_{n,k}} f_{(J^{c},J)}(q_{i,j})l(\eta^{\otimes(n-k) }_{t_{J^{c}}})l^{*}(\J^k\eta^{\otimes k}_{t_J}),
\end{align}
where $f_{(J^c,J)}(q_{ij})=\underset{\lbrace (r,s):\text{ }1\leqslant r\leqslant n-k,1\leqslant s\leqslant k\text{ } i(r)>j(s)\rbrace}\prod  q_{t_{i(r)},t_{j(s)}}$.

Let  $\Gamma^n_T(\h_{\R}, U_t)$ be the subspace of $\Gamma_T(\h_{\R}, U_t)^{\prime\prime}$ spanned by the set of all Wick product of length `n' i.e. the set $\set{W(\xi): \xi\in \h_{\C}^{\otimes n}}$. Denote $\widetilde{\Gamma}_T(\h_{\R},U_t)$ as the linear span of $\{\Gamma^n_T(\h_{\R},U_t)\}_{n\in\N\cup\set{0}}$. Notice that, $\widetilde{\Gamma}_T(\h_{\R},U_t)$ is a w*-dense $*$-subalgebra of $\Gamma_T(\h_{\R},U_t)^{\prime\prime}$ and  $\Gamma_T(\h_{\R},U_t)$ is the C*-completion of $\widetilde{\Gamma}_T(\h_{\R},U_t)$.

\begin{rem}\label{OPP}
The operator $P^{(n)}$ has the following natural decomposition because of the decomposition of $S_n$ into product of $S_{n-k} \times S_{k}$ and $S_n/(S_{n-k} \times S_{k})$ for $0 \leqslant k \leqslant n$ [e.g. Proposition 2.3,~\cite{BKM21}]:
\begin{align}\label{Positive op}
P^{(n)} = (P^{(n-k)} \otimes P^{(k)})R^*_{n-k,k}, \ \ \text{where} \ \ R^*_{n-k,k} = \displaystyle\sum_{\sigma\in S_n/(S_{n-k} \times S_{k})} \pi(\sigma).
\end{align}
Then, by \cite[Lemma IV.3]{BM21} we have the explicit formula for $R^*_{n-k,k}$, given by
\begin{align}\label{Op R}
R^*_{n-k,k}(\eta_{t_1}\otimes\cdots& \otimes\eta_{t_n})= \displaystyle\sum_{J\in\Pi_{n,k}} f_{(J^c,J)}(q_{i,j} )(\eta_{t_{J^c}}^{\otimes(n-k)}\otimes\eta_{t_J}^{\otimes k}), 
\end{align}
where $ \eta=\eta_{t_1}\otimes\cdots\otimes\eta_{t_n}\in\h^{\otimes n}$.\\
In fact, whenever we have a partition $n= n_1+\cdots +n_k$, there exist a unique operator on $\h^{\otimes n}$ such that $P^{(n)} = (P^{(n_1)} \otimes \cdots\otimes P^{(n_k)})R^*_{n_1,n_2,\ldots,n_k}$, which is nothing but the adjoint of the identity map $R_{n_1,n_2,\ldots,n_k} : \h^{\otimes_T ^{n_1}}\otimes\cdots\otimes  \h^{\otimes_T ^{n_k}}\rightarrow \h^{\otimes_T^ n}$. For $n,k,l\in\N$ we have 
\begin{align}\label{rel R}
R^*_{n,k,l}= (Id_n\otimes R^*_{k,l})R^*_{n,k+l} = (R^*_{n,k}\otimes Id_l)R^*_{n+k,l}.
\end{align}
This will follow from the equalities $R_{n,k,l}=R_{n,k+l} (Id_n\otimes R_{k,l}) = R_{n+k,l}(R_{n,k}\otimes Id_l).$
\end{rem}

We fix some more notations here. Given $l,m\in\N$, we can consider multi-indices  $k=(k_1,k_2,\ldots,k_l)\in\set{1,\ldots,m}^l$ as functions $k:\set{1,\ldots,l}\rightarrow \set{1,\ldots,m}.$ For $l\in\N$, we denote by $P(l)$ the set of partitions of the  set $\set{1,\ldots,l}$, which forms a lattice.  In $P(l)$, we have a partial order given by the refinement of partitions, i.e., for $\nu, \nu'\in P(l)$,  $\nu\leqslant\nu'$ iff every element of $\nu$ is a subset of some element in $\nu'$ and $\nu\vee\nu'$ denotes the lattice theoretic join of $\nu$ and $\nu'$. We denote by $\abs{\nu}$ the number of blocks in the partition $\nu$.  For a given multi-index (function)  $k:\set{1,\ldots,l}\rightarrow \set{1,\ldots,m}$, we get a partition in $P(l)$ denoted by ker$ k$ as: $1\leqslant r, s \leqslant l$ belongs to the same block in ker$k$ iff $k_r=k_s$. Let $P_2(l)\subset P(l)$ be the set of all pair partitions of  $\set{1,\ldots,l}$. For $\nu\in P_2(l)$, we denote by $i(\nu)$ the number of crossings in the partition $\nu=\set{(i(r),j(r)):\ 1\leqslant r\leqslant l/2\ \text{and}\ 1\leqslant i(r)< j(r)\leqslant l}$, i.e. the number of pair $(r,s) $ such that $i(r)<i(s)<j(r)<j(s)$. For convenience, we use the notation $[l]$ for the set $\set{1,\ldots,l},$ whenever $l\in\N.$

We have the following formula for the value of $\varphi$ at $W(\xi_{t_1}),\ldots, W(\xi_{t_l})$, for $\xi_{t_i}\in\h^{(t_i)}_{\R}$, from Lemma 3.10 ~\cite{BKM20}:
%\begin{align}
\[
\varphi(W(\xi_{t_1})\cdots W(\xi_{t_l}))=
\begin{cases}
0 & \text{if $l$ is odd}\\
 \displaystyle\sum_{\nu\in P_2(l)} g_{\nu}(q_{i,j})\Pi_{r=1}^{l/2}\inner{\xi_{t_{i(r)}}}{\xi_{t_{j(r)}}}_U & \text{if $l$ is even},
 \end{cases}
\]
%\end{align}
where the summation is over all pair partitions $\nu = \set{(i(r), j(r))}_{1\leqslant r\leqslant l/2}$ of $\set{1, \ldots , l}$ with $i(r) < j(r),$ and $g_{\nu}(q_{ij})$ is given by 
\begin{align*}
g_{\nu}(q_{ij})=\underset{\lbrace (r,s):\text{ }1\leqslant r, s\leqslant l/2, i(r)<i(s)<j(r)<j(s)\rbrace}\prod  q_{t_{i_r},t_{j_s}}.
\end{align*}
Notice that, if  $q_{ij}=q$ for all $i,j\in N$, then $g_{\nu}(q_{ij})=q^{i(\nu)},$ where $i(\nu)=\#\set{ (r,s):\text{ }1\leqslant r,s\leqslant l/2, i(r)<i(s)<j(r)<j(s)}.$

\subsection{Modular theory}
As usual in the Modular theory, let $S_{\varphi}$ be the closure of the operator $x\Omega\rightarrow x^{*}\Omega, x\in \Gamma_{T}(\h_{\R}, U_{t})^{\prime \prime}$, and let $\Delta_{\varphi}, J_{\varphi}$ be the associated modular operator and the modular conjugation respectively. The following formulas are proved in \cite{BKM20}. 
\begin{lem}\label{Modular}
For each $n\geqslant 1$, the following assertions hold:
\begin{enumerate}[label=(\roman*)]
\item $S_{\varphi}(\eta_{1}\otimes\eta_{2}\cdots \otimes\eta_{n})= \eta_{n}\otimes \eta_{{n-1}}\otimes\cdots \otimes \eta_{1} \text{ for }\eta_{1},\dots,\eta_{n}\in \mathcal{H}_{\mathbb{R}}. $
\item $\Delta_{\varphi}(\eta_{1}\otimes\eta_{2}\cdots \otimes\eta_{n})= (A^{-1}\eta_{1})\otimes\cdots\otimes (A^{-1}\eta_{n})$ if $\eta_{1},\dots,\eta_{n}\in \mathcal{H}_{\mathbb{R}}\cap\mathfrak{D}( A^{-1}).$
\item $\Delta_{\varphi}$ restricted on $(\mathcal{H}^{\otimes n},\langle\cdot,\cdot\rangle_{T})$ is the closure of $(A^{-1})^{\otimes n}$ with respect to $\langle \cdot,\cdot\rangle_{T}$.
\item $J_{\varphi}(\eta_{1}\otimes\cdots\otimes\eta_{n})=(A^{-1/2}\eta_{n})\otimes\cdots\otimes(A^{-1/2}\eta_{1}) \text{ if }\eta_{1},\dots,\eta_{n}\in\mathcal{H}_{\mathbb{R}}\cap\mathfrak{D}(A^{-1/2}).$
\end{enumerate}
\end{lem}

We can extend $U_t$ on $\h$ to an one-parameter family of unitary group $\F_T(U_t)$ on $\F_T(\h) $ using Lemma 3.8 of \cite{BKM20}, where $\F_T(U_t)$ is defined as follows:
\begin{align*}
\F_T(U_t)\Omega=\Omega,\ \ \F_T(U_t)\xi_1\otimes\cdots\otimes\xi_n=U_t\xi_1\otimes\cdots\otimes U_t\xi_n
\end{align*}
Then, one can check that
\begin{align*}
\F_T(U_t)s(\xi)\F_T(U_t)^*=s(U_t\xi)\ \text{for}\ \xi\in\h_{\R}.
\end{align*}
Hence, $\alpha_{t}=Ad(\F_T(U_{-t})),$ $t\in\R$ defines a one-parameter group of automorphisms on $\Gamma_T(\h_{\R},U_t)^{\prime\prime}$. In fact, $\set{\alpha_t}_{t\in\R}$ defines the modular automorphism group on $\Gamma_T(\h_{\R},U_t)^{\prime\prime}$ with respect to the vacuum state $\varphi$. This follows from  \cite[Proposition 3.11]{BKM20} and the discussion preceding it.  
 
 The second quantization of mixed $q$-deformed Araki-Woods von Neumann factors assigns every contraction $L: \h_\R\mapsto \mathcal{K}_\R$ between real Hilbert spaces an unital completely positive state preserving map between the corresponding mixed $q$-deformed Araki-Woods von Neumann factors. The following result is proved in Ref.~\cite{BKM20}.

\begin{thm}\label{second quantization}
 Let $\mathcal{K}_{\mathbb{R}}=\oplus_{i=1}^{N_{1}}\mathcal{K}_{\mathbb{R}}^{(i)}$ and $\mathcal{H}_{\mathbb{R}}=\oplus_{i=1}^{N_{2}}\mathcal{H}_{\mathbb{R}}^{(i)}$ be two real Hilbert spaces with one-parameter group of orthogonal transformations $(U_{t})_{t\in\R}, (V_{t})_{t\in\R}$ and let $\mathcal{K}_{\mathbb{R}}^{(i)}$ and $\mathcal{H}_{\mathbb{R}}^{(i)}$ are invariant subspaces for $(U_{t})_{t\in\R}, (V_{t})_{t\in\R}$, respectively, where $N_{2}\geq N_{1}$. Let $-1<q_{ij}<1$ where $1\leqslant i,j\leqslant N_{2}$. Define $T_{1}:\mathcal{K}_{\mathbb{R}}^{(i)}
\otimes\mathcal{K}_{\mathbb{R}}^{(j)}\rightarrow \mathcal{K}_{\mathbb{R}}^{(j)}
\otimes\mathcal{K}_{\mathbb{R}}^{(i)}$ as $\xi \otimes \eta \rightarrow q_{ij}\eta \otimes \xi$ and $T_{2}:\mathcal{H}_{\mathbb{R}}^{(i)}
\otimes\mathcal{H}_{\mathbb{R}}^{(j)}\rightarrow \mathcal{H}_{\mathbb{R}}^{(j)}
\otimes\mathcal{H}_{\mathbb{R}}^{(i)}$ as $\xi \otimes \eta \rightarrow q_{ij}\eta \otimes \xi$. Extend $T_{1}$ to $\mathcal{K}_{\mathbb{R}}\otimes\mathcal{K}_{\mathbb{R}}$ and $T_{2}$ to $\mathcal{H}_{\mathbb{R}}\otimes\mathcal{H}_{\mathbb{R}}$ and denote them again by $T_{1},T_{2}$, respectively. Suppose that $L_{i}:\mathcal{K}_{\mathbb{R}}^{(i)}\rightarrow \mathcal{H}_{\mathbb{R}}^{(i)}$ are contractions and let $L=\oplus_{i=1}^{N_{1}} L_{i}$. Assume that $L_{i}U_{t}=V_{t}L_{i}$ for all $t\in \mathbb{R},\ i\in \lbrace 1,\dots, N_{1}\rbrace $. Then there is a normal ucp map $\Gamma(L):\Gamma_{T_1}(\mathcal{K}_{\R},U_t)^{\prime\prime}\rightarrow \Gamma_{T_{2}}(\h_{\R},V_t)^{\prime\prime}$ extending $s(\eta_{1}\otimes \cdots \eta_{n})\rightarrow s(L\eta_{1}\otimes \cdots L\eta_{n})$ for $\eta_{i}\in \mathcal{H}_{\mathbb{R}}$ for $1\leqslant i\leqslant n.$ Moreover, $\Gamma(L)$ preserves the vacuum states. 
\end{thm}

\subsection{Ultraproducts of von Neumann algebras}

In this section, we recall the notions of ultraproducts of von Neumann algebras from \cite{AH14}. Originally, the ultraproduct of von Neumann algebras was defined for tracial von Neumann algebras. In the case of type III von Neumann algebras, there are two notions of ultraproducts. One is given by Ocneanu, and the other by Raynaud. Ocneanu's ultraproduct is a generalisation  of ultraproducts of tracial von Neumann algebras. However, we restrict ourselves to the Raynaud's ultraproduct as it is more suitable for our purpose.

Before proceeding to the ultraproduct of von Neumann algebras, we briefly recall the ultraproduct of Banach spaces. Fix a free ultrafilter $\omega$ on $\N$. Recall that, $\omega\in\beta\N\setminus\N$, where $\beta\N$ is the Stone-Cech compactification of $\N$. For given a sequence of Banach spaces $(E_n)_{n\in\N}$, consider the space $\ell^{\infty}(\N,E_n)$ of all sequences $(x_n)_n\in \displaystyle\prod_{n\in\N} E_n$ with $\sup_{n\in\N}\norm{x_n}< \infty$. Clearly, $\ell^{\infty}(\N,E_n)$ becomes a Banach space  with the norm given by $\norm{(x_n)_n} = \sup_{n\in\N}\norm{x_n}$, for $(x_n)_n\in\ell^{\infty}(\N,E_n)$. 
Let $\mathcal{I}_{\omega}$ denote the closed subspace of all $(x_n)_n\in\ell^{\infty}(\N,E_n)$ which satisfies
$\lim_{n\rightarrow\omega}\norm{x_n} = 0.$ Then the Banach space ultraproduct $(E_n)_{\omega}$ is defined as the quotient $\ell^{\infty}(\N,E_n)/I_{\omega}$. Any element of $(E_n)_{\omega}$ represented by $(x_n)_n\in\ell^{\infty}(\N,E_n)$ is written as $(x_n)_{\omega}$. If $(H_n)_{n\in\N}$ is a sequence of Hilbert spaces, then the Banach space ultraproduct $H_{\omega}:=(H_n)_{\omega}$ is again a Hilbert space with the inner product defined by $\inner{(\xi_n)_{\omega}}{(\eta_n)_{\omega}}:= \lim_{n\rightarrow\omega}\inner{\xi_n}{\eta_n}$, for $(\xi_n)_{\omega},(\eta_n)_{\omega}\in (H_n)_{\omega}$.

Suppose $(M_n,\varphi_n)$ is a sequence of von Neumann algebras equipped with normal faithful states. For each $n\in\N$, we can realize $M_n\subset \B(H_n)$, using GNS representation of $M_n$ with respect to $\varphi_n$.  Let $(M_n)_{\omega}$ be the Banach space ultraproduct of $(M_n)_{n\in\N}$. Note that, $(M_n)_{\omega}$ is $C^*$-algebra with respect to the norm $\norm{(x_n)_{\omega}}=\displaystyle\lim_{n\rightarrow\omega}\norm{x_n}$. Also, let $H_{\omega}$ be the Banach space ultraproduct of the sequence of GNS spaces $(H_n)_{n\in\N}$. Consider the  diagonal action $\Theta:(M_n)_{\omega}\rightarrow \B(H_{\omega})$ defined by 
\begin{equation}
\Theta((a_n)_{\omega})(\xi_n)_{\omega}:=(a_n\xi_n)_{\omega}.
\end{equation}
It can be easily checked that this action is a well-defined $*$-homomorphism and 
\begin{equation*}
\norm{\Theta((a_n)_{\omega})}=\displaystyle\lim_{n\rightarrow\omega}\norm{a_n}=\norm{(a_n)_{\omega}}. 
\end{equation*}
Hence, $\Theta$ is an injective $*$-homomorphism. The Raynaud ultraproduct is defined as the weak closure of $\Theta((M_n)_{\omega})$ inside $\B(H_{\omega})$ and  denoted by $\prod^{\omega}(M_n,\varphi_n)$. The ultraproduct state on the Raynaud ultraproduct, denoted by $\varphi_{\omega}$, is a vector state given by the ultraproduct of the cyclic vectors for the GNS representations of algebras $M_n$, i.e.
\begin{equation} \varphi_{\omega}(x):=\inner{\xi_{\omega}}{x\xi_{\omega}},\quad \text{for}\quad x\in\prod^{\omega}(M_n,\varphi_n),
\end{equation}
 where $\xi_{\omega}:=(\xi_n)_{\omega}\in(\h_n)_{\omega}$.

 \subsection{Radial Multipliers}
 In this article, the main goal is to find the completely bounded norm of a special class of completely bounded linear maps on mixed $q$-deformed Araki–Woods factors, called radial multipliers.
 
 \begin{defn}
  Let $\varphi :\N\cup\set{0}\rightarrow \C$ be a bounded function. The linear map $m_{\varphi}:\tilde{\Gamma}_T(\h_{\R},U_t)\rightarrow\tilde{\Gamma}_T(\h_{\R},U_t)$ given by
 \begin{align*}
 m_{\varphi}(W(\xi))= \varphi(n)W(\xi) \ \text{ for}\ \ \xi\in(\h_{\C})^{\otimes n}
 \end{align*}
 is called the radial multiplier. If $m_{\varphi }$ extends to a completely bounded map  $m_{\varphi}:\Gamma_T(\h_{\R},U_t)\rightarrow\Gamma_T(\h_{\R},U_t)$, we call $m_{\varphi}$ a completely bounded radial multiplier on $\Gamma_T(\h_{\R}, U_t)^{\prime\prime}$. 
 \end{defn}
 
We can consider the same radial multiplier on the mixed $q$-Gaussian algebras. The idea is to find the cb norm of these radial multipliers on the mixed $q$-Gaussian von Neumann algebras and use the same to find the cb norm of the same radial multiplier on the mixed $q$-deformed Araki-Woods von Neumann factors. We prepare ourselves to find the cb norm of the radial multipliers on the mixed $q$-Gaussian algebras here.

Notice that, in the construction of the mixed $q$-deformed Araki-Woods von Neumann factor, if the orthogonal family $\set{U_t: t\in \R}$ is trivial, the von Neumann algebra  $\Gamma_T(\h_{\R}, U_t)^{\prime\prime}$ will be a mixed $q$-Gaussian von Neumann algebra, which we denote as $\Gamma_Q(\h_{\R})^{\prime\prime}$. The vacuum state $\varphi$ becomes a trace and will be denoted by $\varphi_Q$. The inner product on the Fock space will be denoted by $\inner{\cdot}{\cdot}_Q$. %The closure of $\h^{\otimes n}$ with respect to the inner product  $\inner{\cdot}{\cdot}_Q$ will be denoted by $\h^{\otimes^ n_Q}$.

For $n\in\N\cup\set{0}$, let $F_n : \Gamma_T(\h_{\R},U_t)^{\prime\prime}\rightarrow \Gamma_T(\h_{\R},U_t)^{\prime\prime}$ be the projection defined by $F_n(W(\xi)) =\delta_{m,n} W(\xi)$ for $\xi\in\h_{\C}^{\otimes m}$. Notice that, for $n\in\N\cup\set{0}$ the radial multipliers $m_{\varphi_n}$ corresponding to the map $\varphi_n:\N\cup\set{0}\rightarrow\C$ defined by $\varphi_n(m)=\delta_{m,n}$ is nothing but $F_n$. So, our aim is to find the cb norm of $F_n$ on $\Gamma_T(\h_{\R}, U_t)^{\prime\prime}$.

\section{Embedding of mixed deformed $q$-Araki-Woods factors into ultraproduct }

Recall that, to establish the $w^*$-complete metric approximation property of the mixed $q$-deformed Araki-Woods factor, we have to find a  net of completely bounded, $w^*$-continuous  maps and  $(\phi_i)$, whose cb norms are less than or equal to $1$ and $ \lim_i \phi_i(x) = x $ for all $x \in \Gamma_{T}(\h_{\R}, U_{t})^{\prime \prime}$.  In this setting to find such maps our strategy is to use the maps $(F_n)_{ n\geqslant 0 }$ (as considered in the end of Section.2),  and  a family of second quantization maps given by Theorem \ref{second quantization}.  As the  second quantization maps are completely contractive,  hence, to get  a family of maps whose completely bounded norms are less than or equal to $1$, we need to estimate the completely bounded norms of $F_n$ $(n\geqslant 0)$.

In this section, we provide an embedding of mixed $q$-deformed Araki-Woods von Neumann factor inside the ultraproduct of tensor products of some suitable mixed $q$-Gaussian and $q$-Araki-Woods von Neumann factors using the techniques of \cite{JXu07} and \cite{ABW18}. In \cite{JXu07}, an upper bounds for the cb norms of $F_n$'s when defined on $\Gamma_Q(\h_{\R})^{\prime\prime}$ were already provided. We will use these bounds together with the embedding to get upper bounds for the same maps $F_n$ $(n\geqslant 0)$ when defined on the mixed $q$-deformed Araki-Woods von Neumann factor.

Let $\Gamma_T(\h_{\R},U_t)^{\prime\prime}$ be a fixed mixed $q$-deformed Araki-Woods factor for some  $Q=(q_{ij})_{ij}$ such that,  $\max_{i,j}   |q_{ij}|   < 1$. We can find some $q$ such that $\max_{i,j}  |q_{ij}|   < q < 1.$ Let $\tilde{Q} :=(\tilde {q}_{ij})_{ij}$ such that $q_{ij}=q\tilde{q}_{ij} $ for $i,j\in N.$ Notice that $Q=q\tilde{Q}$ and  $\max_{i,j}  |\tilde{q}_{ij}|   < 1$. For any $m\in \N$ and $N\geq m$, we consider the decomposition of $\R^m$ as:
 \begin{align*}
  \R^m=\underset{m}{\underbrace{\R\oplus\cdots\oplus\R}}\oplus\underset{N-m}{\underbrace{\set{0}\oplus\cdots\oplus\set{0}}}   
 \end{align*}
 For $m\in \N$ and $m>N$, we consider a new matrix, which we denote by $\tilde{Q}$ again, by putting $q_{ij}=0$ whenever $i$ or $j>N$. With these convention, we consider the mixed $q$-Gaussian algebra $\Gamma_{\tilde{Q}}(\R^m)^{\prime\prime}$ associated to the matrix $\tilde{Q}$. Let $\Gamma_q(\h_{\R}\otimes \R^m, U_{t}\otimes I_m)^{\prime\prime}$ be the $q$-Araki-Woods factor, where $I_m$ denotes the identity on $\R^m$. For each $m$, fix an orthonormal basis $(e_1, \ldots , e_m)$ of $\R^m$ and for $\xi\in\h$, we define
\begin{align*}
u_m(\xi):= m^{-\frac{1}{2}}\displaystyle\sum_{k=1}^m W(e_k)\otimes W(\xi\otimes e_k)\in \Gamma_{\tilde{Q}}(\R^m)^{\prime\prime}\overline{\otimes} \Gamma_q(\h_{\R}\otimes \R^m,U_t\otimes I_m)^{\prime\prime}.
 \end{align*}
For a non-principal ultrafilter $\omega$ of $\N$, consider the Raynaud's ultraproduct 
$$\mathcal{N}:=\displaystyle\prod^{\omega}\Gamma_{\tilde{Q}}(\R^m)^{\prime\prime}\overline{\otimes} \Gamma_q(\h_{\R}\otimes \R^m, U_t\otimes I_m)^{\prime\prime} $$ and let $(\varphi_{\tilde{Q}}\otimes \varphi_q)_{\omega}$ denote the ultraproduct state on $\mathcal{N}$. 
\begin{rem}\label{norm estimation}
 Note that for any operators $x_1, x_2,\ldots, x_m\in\B(H)$ and $1$ being the identity of $B(H)$, where $H$ is some Hilbert space and $m\in\N$, we have $\displaystyle\sum_{k=1}^m x_k^*x_k\leqslant\displaystyle\sum_{k=1}^m \norm{x_k}^2$. Further, suppose $ \norm{x_1} = \norm{ x_2} = \cdots = \norm{ x_m} $, then it follows that $\displaystyle\sum_{k=1}^m x_k^*x_k\leqslant m \norm{x_1}^2$. Similarly, we will have $\displaystyle\sum_{k=1}^m x_kx_k^*\leqslant m \norm{x_1}^2$.   
\end{rem}

\begin{lem}
For all $\xi\in\h_{\C}$, $(u_m(\xi))_{\omega}\in\mathcal{N}$.
\end{lem}
\begin{proof}
To show $(u_m(\xi))_{\omega}\in\mathcal{N}$, it is enough to conclude that the sequence $(u_m(\xi))_{m\in\N}$ is bounded. By definition, it follows
\begin{align}\label{defn of u}
    \norm{u_m(\xi)} = \norm{m^{-\frac{1}{2}}\displaystyle\sum_{k=1}^m W(e_k)\otimes W(\xi\otimes e_k) },
\end{align}
where $W(e_k)\in\Gamma_{\tilde{Q}}(\R^m)^{\prime\prime}$, which is also a mixed $q$-deformed Araki-Woods factor. 
Note that, 
$$m^{-\frac{1}{2}}\displaystyle\sum_{k=1}^m   W(\xi\otimes e_k)\otimes e_k \in \B(\F_q(\h\otimes \C^m))\otimes_{min} \C^m.$$
Hence, we can use the right hand side inequality of \cite[Theorem 3.10]{BKM21}  for $n=1$ and $\C^m$ with $K=\F_q(\h\otimes \C^m)$, we  write 
\begin{align*}
&\norm{(Id\otimes s)(m^{-\frac{1}{2}}\displaystyle\sum_{k=1}^m   W(\xi\otimes e_k)\otimes e_k)}_{min}\\ &\leqslant  2C_q \max\set{\norm{m^{-\frac{1}{2}}\displaystyle\sum_{k=1}^m W(\xi\otimes e_k)\otimes R_{1,0}^*(e_k)},\ \norm{m^{-\frac{1}{2}}\displaystyle\sum_{k=1}^m  W(\xi\otimes e_k)\otimes R_{0,1}^*(e_k)}},\\
%&=\norm{m^{-\frac{1}{2}}\displaystyle\sum_{k=1}^m W(e_k)\otimes W(\xi\otimes e_k) } \\
\end{align*}
where $s$ defines the generalized wick product formula for vectors in $(\C^m)^{\otimes^n_T}$ given by \cite[Lemma 3.7]{BKM21}, $Id$ denotes the identity on $\B(K)$, and the two norms in the right hand side are calculated in $\B(K)\otimes_{min}(\C^m)_c$ and $\B(K)\otimes_{min}(\C^m)_r$, respectively. Since $e_k\in\R^m$, it is clear from \cite[Lemma 3.7]{BKM21} that $s(e_k)= W(e_k).$ Therefore, the above inequality now becomes
\begin{align*}
&\norm{m^{-\frac{1}{2}}\displaystyle\sum_{k=1}^m   W(\xi\otimes e_k)\otimes W(e_k)}\\ &\leqslant  2C_q \max\set{\norm{m^{-\frac{1}{2}}\displaystyle\sum_{k=1}^m W(\xi\otimes e_k)\otimes R_{1,0}^*(e_k)},\ \norm{m^{-\frac{1}{2}}\displaystyle\sum_{k=1}^m  W(\xi\otimes e_k)\otimes R_{0,1}^*(e_k)}}.
%&=\norm{m^{-\frac{1}{2}}\displaystyle\sum_{k=1}^m W(e_k)\otimes W(\xi\otimes e_k) } \\
\end{align*}

Hence, using commutativity of tensor product norm in Eq.\ref{defn of u}, we have
\begin{align*}
&\norm{u_m(\xi)}\\
% &\norm{m^{-\frac{1}{2}}\displaystyle\sum_{k=1}^m  W(\xi\otimes e_k) \otimes W(e_k)}\\
% &= \norm{(Id \otimes W) (m^{-\frac{1}{2}}\displaystyle\sum_{k=1}^m  W(\xi\otimes e_k) \otimes e_k)}_{ \text{ min}}\\
&\leqslant  m^{-\frac{1}{2}}2C_q \max\set{\norm{\displaystyle\sum_{k=1}^m W(\xi\otimes e_k)\otimes R_{1,0}^*(e_k)},\ \norm{\displaystyle\sum_{k=1}^m  W(\xi\otimes e_k)\otimes R_{0,1}^*(e_k)}}.\\
\end{align*}
 Recall from the Eq.\ref{Op R} that both $R^*_{1,0}(e_k)$ and $R^*_{0,1}(e_k)$ are equal to $e_k$, for all $1\leqslant k\leqslant m$. Clearly, we can identify $(\C^m)_c$ with $M_{m,1}(\C)$ and hence $\displaystyle\sum_{k=1}^m W(\xi\otimes e_k)\otimes R_{1,0}^*(e_k)$ can be thought as a `$m\times 1 $' column matrix with entries $(W(\xi\otimes e_k))_{1\leqslant k\leqslant m}\subseteq\B(K)$. Hence,
\begin{align}\label{col}
    \norm{\displaystyle\sum_{k=1}^m W(\xi\otimes e_k)\otimes e_k}_{\B(K)\otimes_{min}(\C^m)_c}  = \norm{ \begin{pmatrix}
    W(\xi\otimes e_1)\\
    \vdots\\
    W(\xi\otimes e_m)
    \end{pmatrix}}_{M_{m,1}(\B(K))}\\
    \nonumber = \norm{\displaystyle\sum_{k=1}^m W(\xi\otimes e_k)^*W(\xi\otimes e_k)}^{\frac{1}{2}}.
\end{align}
In a similar manner, we identify $(\C^m)_r$ with $M_{1,m}(\C)$ and subsequently, we can see $\displaystyle\sum_{k=1}^m W(\xi\otimes e_k)\otimes R_{0,1}^*(e_k)$ as a `$1\times m $' row matrix with entries $(W(\xi\otimes e_k))_{1\leqslant k\leqslant m}\subseteq\B(K)$. This implies
\begin{align}\label{row}
    \norm{\displaystyle\sum_{k=1}^m W(\xi\otimes e_k)\otimes e_k}_{\B(K)\otimes_{min}(\C^m)_r} & = \norm{ \begin{pmatrix}
    W(\xi\otimes e_1) & \cdots & W(\xi\otimes e_m)
    \end{pmatrix}}_{M_{1,m}(\B(K))}\\ & 
    \nonumber = \norm{\displaystyle\sum_{k=1}^m W(\xi\otimes e_k)W(\xi\otimes e_k)^*}^{\frac{1}{2}}.
\end{align}

% \begin{align*}
% &\norm{u_m(\xi)}\\
% &\leqslant  2C_q m^{-\frac{1}{2}}\max\set{\norm{\displaystyle\sum_{k=1}^m  W(\xi\otimes e_k)\otimes e_k},\ \norm{\displaystyle\sum_{k=1}^m  W(\xi\otimes e_k)\otimes e_k}}\\
% &\leqslant  2C_q m^{-\frac{1}{2}}\\
% &\qquad\max\set{\norm{\displaystyle\sum_{k=1}^mW(\xi\otimes e_k)^*W(\xi\otimes e_k)}^{\frac{1}{2}},\norm{\displaystyle\sum_{k=1}^mW(\xi\otimes e_k)W(\xi\otimes e_k)^*}^{\frac{1}{2}}}\\
% &\\
% %&\leqslant\max\set{\norm{\displaystyle\sum_{k=1}^mW(\xi\otimes e_k)^*W(\xi\otimes e_k)}^{\frac{1}{2}},\norm{\displaystyle\sum_{k=1}^mW(\xi\otimes e_k)W(\xi\otimes e_k)^*}^{\frac{1}{2}}}\\
% %&\leqslant 2C_q\norm{W(\xi\otimes e_1)}.
% \end{align*}
Hence, it follows from Remark \ref{norm estimation}, and Eq.\ref{col} and Eq.\ref{row}  that
\begin{align*}
   &\norm{u_m(\xi)}\\
   &\leqslant  m^{-\frac{1}{2}}2C_q \max\set{\norm{\displaystyle\sum_{k=1}^m W(\xi\otimes e_k)\otimes R_{1,0}^*(e_k)},\ \norm{\displaystyle\sum_{k=1}^m  W(\xi\otimes e_k)\otimes R_{0,1}^*(e_k)}}\\
   &\leqslant 2C_q\norm{W(\xi\otimes e_1)}.
\end{align*}
\end{proof}

Recall that $\Gamma_T(\h_{\R},U_t)^{\prime\prime}$ is the von Neumann algbera generated by $\set{s(\xi):\xi\in\h_{\R}}$. Let $\set{f_j}_{j\in\Lambda}$ be a orthonormal basis of $\h_{\R}$. It can also be seen that $\set{s(f_j):\ j\in\Lambda}$ generates the von Neumann algebra $\Gamma_T(\h_{\R},U_t)^{\prime\prime}$. We denote $\mathcal{A}:=*\text{-}alg\set{s(f_j)}_{ j\in\Lambda}$ and $B:=*\text{-}alg\set{(u_m(f_j))_{\omega}:\ j\in\Lambda}\subset \mathcal{N}$. Let $\Tilde{B}$ be the $w^*$ closure of $B$, which is a von Neumann subalgebra of $\mathcal{N}.$ \\

Consider the ultraproduct state $(\varphi_{\tilde{Q}}\otimes \varphi_q)_{\omega}$   on $\mathcal{N}$ and it is known that the ultraproduct state is not faithful, in general. So, let $p$ be the support projection of the ultraproduct state $(\varphi_{\tilde{Q}}\otimes \varphi_q)_{\omega}$, which belongs to $\mathcal{N}$. With these notations,  we have the following results.  
\begin{lem}\label{Product}
    For all $x,y\in \Tilde{B}$, $pxyp=pxpyp$. 
    % , where $p= supp((\varphi_{\tilde{Q}}\otimes \varphi_q)_{\omega})$. \color{red}{ what is p} 
\end{lem}
\begin{proof}
    It follows from \cite[Lemma 4.1]{Nou06} that to establish our claim it is enough to show that for all $x\in B$, there is a representative $(x_m)_{m\in\N}$ of $x$ such that for all $m\in\N$, $x_m$ is entire for $(\sigma_t^m)_{t\in\R}$ and $(\sigma_{-i}^m(x_m))_{m\in\N}$ is uniformly bounded, where $\sigma_t^m$ is the modular automorphism group of $\Gamma_{\tilde{Q}}(\R^m)^{\prime\prime}\overline{\otimes} \Gamma_q(\h\otimes \C^m, U_t\otimes I_m)^{\prime\prime}$. Since $B$ is the $*$-algebra generated by $(u_m(f_j))_{\omega}$ $(j\in\Lambda)$, it suffices to check that $(u_m(f_j))_{\omega}$ is analytic for the modular automorphism of $\mathcal{N}$ and $(\sigma_{-i}^m(u_m(f_j)))_{m\in\N}$ is uniformly bounded. Note that the modular automorphism group $(\sigma_t^m)_{t\in\R}$ on $\Gamma_{\tilde{Q}}(\R^m)^{\prime\prime}\overline{\otimes} \Gamma_q(\h\otimes \C^m,U_t\otimes I_m)^{\prime\prime}$ is given by
    \begin{align}\label{tensor auto}
        \sigma_t^m =id_{\Gamma_{\tilde{Q}}(\R^m)} \otimes\sigma_t^{\varphi_q}\ \ \ (t\in\R).
    \end{align}

 Recall that $A$ is the analytic generator of $(U_t)$. Now  suppose  $$\h_{\C}^{an} =\cup_{\lambda >1}\chi_{[1/\lambda,\lambda]}(A)\h_{\C}, $$
 where $\chi_{[1/\lambda,\lambda]}(A)$ denotes the spectral projection of the analytic generator $A$ corresponding to the interval $[1/\lambda,\lambda]$.
 Further, we observe that, $$\J\chi_{[1/\lambda,\lambda]}(A) = \chi_{[1/\lambda,\lambda]}(A^{-1})\J = \chi_{[1/\lambda,\lambda]}(A)\J, $$ 
 where  $\J$ is the  complex conjugation on $\h_{\C}$ and thus, we have  $\J\h_{\C}^{an} = \h_{\C}^{an}.$ Since $A$ is non-degenerate, we further notice that  $$\overline{\cup_{\lambda >1}\chi_{[1/\lambda,\lambda]}(A)\h_{\C}}=\chi_{(0,\infty)}(A)\h_{\C}=\h_{\C}.$$
    This implies that $\h_{\C}^{an}\subset\h_{\C}$ is a dense subspace. Moreover, it follows from the proof of \cite[Theorem 3.1]{BH16} that each $\xi\in\h_{\C}^{an}$  is analytic for the action of the unitary group $U_t = A^{it}$ $(t\in\R)$. Since $\xi\in\h_{\C}^{an}$, there exist $\lambda > 1$  such that $\xi\in \chi_{[1/\lambda,\lambda]}(A)\h_{\C} $.
    Clearly, $\chi_{[1/\lambda,\lambda]}(A)\xi=\xi$. Let $\tilde{A}=A\chi_{[1/\lambda,\lambda]}(A)$, which is a bounded invertible operator with $\norm{\tilde{A}}\leqslant\lambda$. Therefore, for all $z\in\C$, $\tilde{A}^{z}$ is bounded. So, $\xi\in\mathcal{D}_{\tilde{A}^z}$ for all $z\in\C$. This implies, the map $t\mapsto U_t\xi=A^{it}\chi_{[1/\lambda,\lambda]}(A)\xi = \tilde{A}^{it}\xi$ has an entire analytic extension (see also \cite[Remark 2.6]{BM17}). 
    
    This implies $W(\xi)$ is analytic for the modular automorphism group $(\sigma^{\varphi_q}_t)_{t\in\R}$ and $\sigma^{\varphi_q}_{z}(W(\xi))=W(A^{-iz}\xi)$ for $z\in\C$.

%\begin{align*}
 %\\
%&\sigma_t^{\varphi_q}(W(\xi\otimes f))=W(A^{-it}\xi\otimes f) \ \ \ (\xi\in\h_{\C} , f\in\C^m).
%\end{align*}
Therefore, it follows from Eq.\ref{tensor auto} that if $\xi\in\h^{an}$ and $f \in\C^m $, then elements $W(\xi\otimes f)$ and $u_m(\xi)$ are analytic for their respective modular groups and 
\begin{align*}
\sigma_z^m(u_m(\xi))&=m^{-\frac{1}{2}}\displaystyle\sum_{k=1}^m W(e_k)\otimes \sigma_z^{\varphi_q}(W(\xi\otimes e_k))\\
&=m^{-\frac{1}{2}}\displaystyle\sum_{k=1}^m W(e_k)\otimes W(A^{-iz}\xi\otimes e_k)\\
&=u_m(A^{-iz}\xi).
\end{align*}
Since, the sequence $(u_m(\xi))_{m\in\N}$ is bounded for all $\xi\in\h_{\C}$ and an upper bound is given in the previous lemma, we write 
\begin{align*}
\sup_m\norm{\sigma^m_{-i}(u_m (\xi)}=\sup_m\norm{u_m(A^{-1}\xi)}\leqslant 2C_q\norm{W(A^{-1}\xi\otimes e_1)}.
\end{align*}

\end{proof}

\begin{thm}\label{Embedding}
There exists a unique normal $*$-isomorphism $\Theta:\Gamma_T(\h_{\R},U_t)^{\prime\prime}\mapsto p\Tilde{B}p$, which preserves the states i.e. $p(\varphi_{\tilde{Q}}\otimes \varphi_q)_{\omega}p\circ\Theta=\varphi$ and $\Theta(s(f_i))=p(u_m(f_i))_{\omega}p ,$ for all $i\in\Lambda$. 
\end{thm}

\begin{proof}
% The linearity of map $\Theta$ defined by $\Theta(s(f_i))=p(u_m(f_i))_{\omega}p$ is trivial and it preserves the product because of Lemma \ref{Product}. 
Recall that $\Gamma_T(\h_{\R},U_t)^{\prime\prime}$ is generated by the set of self-adjoint elements $\set{s(f_j):\  j\in\Lambda}$ and the generators of the von Neumann subalgebra $\tilde{B}$ of $\mathcal{N}$ is the set $\set{(u_m(f_j))_{\omega}:\ j\in\Lambda}.$ 
Then, it follows from \cite[Proposition 2.1]{BISS14} that to show there exists an isomorphism $\Theta$, it is enough to show  for all $l\in\N$, and $t_1,\ldots,t_l\in\Lambda$
\begin{align}\label{state-preserving}
(\varphi_{\tilde{Q}}\otimes\varphi_q)_{\omega}((u_m(f_{t_1}))_{\omega}\cdots (u_m(f_{t_l}))_{\omega})= \varphi(W(f_{t_1})\cdots W(f_{t_l})).
\end{align}
Note that 
\begin{align*}
& (\varphi_{\tilde{Q}}\otimes\varphi_q)_{\omega}((u_m(f_{t_1}))_{\omega}\cdots (u_m(f_{t_l}))_{\omega}\\
&=\lim_{m\rightarrow\omega} (\varphi_{\tilde{Q}}\otimes\varphi_q)(u_m(f_{t_1})\cdots u_m(f_{t_l}))\\
&=\lim_{m\rightarrow\omega}\frac{1}{m^{l/2}}\displaystyle\sum_{k_1,\ldots,k_l=1}^m\varphi_{\tilde{Q}}(W(e_{k_1})\cdots W(e_{k_l}))\varphi_q(W(f_{t_1}\otimes e_{k_1})\cdots W(f_{t_l}\otimes e_{k_l}))
\end{align*}
For odd $l$, recall from  ~\cite[Lemma 3.10]{BKM20}  that $\varphi_{\tilde{Q}}(W(e_{k_1})\cdots W(e_{k_l}))=0$, implying
$(\varphi_{\tilde{Q}}\otimes\varphi_q)(u_m(f_{t_1})\cdots u_m(f_{t_l}))= 0$. Also, $\varphi(W(f_{t_1})\cdots W(f_{t_l}))=0$. This establishes the Eq.\ref{state-preserving}, for all $l\in\N$ odd.

Hence, we assume $l$ is even. Then, for fixed $m\in\N$, we have,
\begin{align*}
&(\varphi_{\tilde{Q}}\otimes\varphi_q)(u_m(f_{t_1})\cdots u_m(f_{t_l}))\\
&=\frac{1}{m^{l/2}}\displaystyle\sum_{k_1,\ldots,k_l=1}^m\varphi_{\tilde{Q}}(W(e_{k_1})\cdots W(e_{k_l}))\varphi_q(W(f_{t_1}\otimes e_{k_1})\cdots W(f_{t_l}\otimes e_{k_l}))\\
&=\frac{1}{m^{l/2}}\displaystyle\sum_{k_1,\ldots,k_l=1}^m\left(\displaystyle\sum_{\nu\in P_2(l)} g_{\nu}(\tilde{q}_{i,j})\displaystyle\prod_{r=1}^{l/2}\inner{e_{k_{i(r)}}}{e_{k_{j(r)}}} \right)\\
&\qquad\qquad\qquad\qquad\qquad\qquad\qquad\qquad\qquad\left(\displaystyle\sum_{\nu'\in P_2(l)}q^{i(\nu')}\inner{f_{t_r}\otimes e_{k_r}}{f_{t_s}\otimes e_{k_s}}_U\right)\\
&=\frac{1}{m^{l/2}}\displaystyle\sum_{k_1,\ldots,k_l=1}^m\left(\displaystyle\sum_{\substack{\nu\in P_2(l)\\ \text{ker} k\geq \nu}} g_{\nu}(\tilde{q}_{i,j}) \right)\left(\displaystyle\sum_{\substack{\nu'\in P_2(l)\\ \text{ker} k\geq \nu'}}q^{i(\nu')}\displaystyle\prod_{(r,s)\in\nu'}\inner{f_{t_r}}{f_{t_s}}_U\right)\\
&=\displaystyle\sum_{\nu,\nu'\in P_2(l)}g_{\nu}(\tilde{q}_{i,j})q^{i(\nu')}\displaystyle\prod_{(r,s)\in\nu'}\inner{f_{t_r}}{f_{t_s}}_U\displaystyle\sum_{\substack{k_1,\ldots,k_l=1\\ \text{ker}k\geq\nu,\text{ker}k\geq\nu'}}^m \frac{1}{m^{l/2}}\\
&=\displaystyle\sum_{\nu,\nu'\in P_2(l)}g_{\nu}(\tilde{q}_{i,j})q^{i(\nu')}\displaystyle\prod_{(r,s)\in\nu'}\inner{f_{t_r}}{f_{t_s}}_U \frac{1}{m^{l/2-|\nu\vee\nu'|}}.
\end{align*} 
Hence, 
\begin{align*}
\lim_{m\rightarrow\infty}(\varphi_{\tilde{Q}}\otimes\varphi_q)&(u_m(f_{t_1})\cdots u_m(f_{t_l}))\\
&=\displaystyle\sum_{\nu,\nu'\in P_2(l)}g_{\nu}(\tilde{q}_{i,j})q^{i(\nu')}\displaystyle\prod_{(r,s)\in\nu'}\inner{f_{t_r}}{f_{t_s}}_U \left(\lim_{m\rightarrow\infty}\frac{1}{m^{l/2-|\nu\vee\nu'|}}\right)\\
&=\displaystyle\sum_{\nu,\nu'\in P_2(l)}g_{\nu}(\tilde{q}_{i,j})q^{i(\nu')}\displaystyle\prod_{(r,s)\in\nu'}\inner{f_{t_r}}{f_{t_s}}_U \delta_{\nu,\nu'}\\
&=\displaystyle\sum_{\nu\in P_2(l)}g_{\nu}(\tilde{q}_{i,j})q^{i(\nu)}\displaystyle\prod_{(r,s)\in\nu}\inner{f_{t_r}}{f_{t_s}}_U \\
&=\displaystyle\sum_{\nu\in P_2(l)}g_{\nu}(q_{i,j})\displaystyle\prod_{(r,s)\in\nu}\inner{f_{t_r}}{f_{t_s}}_U\\
&=\varphi(W(f_{t_1})\cdots W(f_{t_l})).
\end{align*}
This establishes the following: 
\begin{align*}
(\varphi_{\tilde{Q}}\otimes\varphi_q)_{\omega}((u_m(f_{t_1}))_{\omega}\cdots (u_m(f_{t_l}))_{\omega})= \varphi(W(f_{t_1})\cdots W(f_{t_l})),\ \text{for all}\ t_1,\ldots,t_l\in\Lambda.
\end{align*}

Then, the result follows from \cite[Proposition 2.1]{BISS14}.
\end{proof}
The following theorem establishes that the subalgebra containment $p\Tilde{B}p\subseteq\mathcal{N}$ is with expectation.  
\begin{lem}
    There exists a normal state-preserving conditional expectation from $\mathcal{N}$ onto $p\Tilde{B}p$.
\end{lem}

\begin{proof}
   Since $p\Tilde{B}p$ is a subalgebra of $p\mathcal{N}p$, it follows from \cite{Ta72} that there is normal state-preserving conditional expectation from $p\mathcal{N}p$ onto $p\Tilde{B}p$ if and only if $p\Tilde{B}p$ is stable by the modular automorphism group of the ultraproduct state $(\varphi_{\tilde{Q}}\otimes \varphi_q)_{\omega}$. By \cite[Theorem 2.1]{Ra02}, the modular automorphism group with respect to the ultraproduct state is given by $(\sigma_t^m)_{\omega}$ $(t\in\R)$, where $\set{\sigma_t^m}_{t\in\R}$ is the modular automorphism of $\Gamma_{\tilde{Q}}(\R^m)^{\prime\prime}\overline{\otimes} \Gamma_q(\h_{\R}\otimes \R^m)^{\prime\prime}$. Recall from Eq.\ref{tensor auto} that $\sigma_t^m=id_{\Gamma_{\tilde{Q}}(\R^m)} \otimes\sigma_t^{\varphi_q}, \ (t\in\R)$ and for all $m$, $\sigma_t^m(u_m(\xi))=u_m(A^{-it}\xi)\ (\xi\in\h_{\C})$. Since $p=supp(\varphi_{\tilde{Q}}\otimes\varphi_q)_{\omega}$, we have
   \begin{align*}
        (\sigma_t^m)_{\omega}(p(u_m(\xi))_{\omega}p)&=
       p((\sigma_t^m(u_m(\xi)))_{\omega})p\\ &=p((u_m(A^{-it}\xi))_{\omega})p=p\Theta(W(A^{-it}\xi))p\in p\tilde{B}p. 
   \end{align*}
   Hence, our assertion follows as the map $x\mapsto pxp$ gives a normal state-preserving conditional expectation from $\mathcal{N}$ onto $p\mathcal{N}p$.\\
\end{proof}

For $l\in\N$ and $\xi_{t_1},\ldots,\xi_{t_l}\in\h_{\C}$, we define the following element of $\mathcal{N}$. 
\begin{align*}
W^{\omega}(\xi_{t_1}&\otimes\cdots\otimes\xi_{t_l})\\ \nonumber
&:=\left(m^{\frac{-l}{2}}\displaystyle\sum_{\substack{k:[l]\mapsto [m]\\ injective}}^m W(e_{k_1})\cdots W(e_{k_l})\otimes W(\xi_{t_1}\otimes e_{k_1})\cdots W(\xi_{t_l}\otimes e_{k_l})\right)_{\omega}.\\
\end{align*}

We consider only those terms when $m\geqslant l$. The representative of $W^{\omega}(\xi_{t_1}\otimes\cdots\otimes\xi_{t_l})$ can be taken as the sequence with first $l$ entries as `0', which are finitely many and rest of the terms are given by the above expression.

\begin{thm}\label{image}
For any $l\in\N$, let $\xi_{t_1},\ldots,\xi_{t_l}\in\h_{\C}$. Then
\begin{align*}
 \Theta (W(\xi_{t_1}\otimes\cdots\otimes\xi_{t_l}))=W^{\omega}(\xi_{t_1}&\otimes\cdots\otimes\xi_{t_l}).\\
 \end{align*}
\end{thm}

Before proving the Theorem \ref{image}, we go through some intermediate lemmas. Suppose $k : [l+1]\rightarrow [m]$ be a  multi-index, and then for  simplicity of notations, we introduce the following: 
\begin{enumerate}
    \item 
$D_{1,i, k}(m):=W(e_{k_1} \otimes\cdots\otimes e_{k_{l+1}})\otimes W((\xi_{t_1}\otimes e_{k_1})\otimes\cdots\otimes (\xi_{t_{l+1}}\otimes e_{k_{l+1}}))$ \\
\item 
$D_{2,i, k}(m):= W(e_{k_1} \otimes\cdots\otimes e_{k_{l+1}})\otimes W(\xi_{t_1}\otimes e_{k_1})\cdots\\
\null\hfill \cdots \widehat{W(\xi_{t_i}\otimes e_{k_i})}\cdots W(\xi_{t_{l+1}}\otimes e_{k_{l+1}})$,\\
\item 
$D_{3,i,k}(m):= W(e_{k_1})\cdots \widehat{W(e_{k_i})}\cdots W(e_{k_{l+1}}) \otimes W((\xi_{t_1}\otimes e_{k_1})\otimes\cdots\\\null\hfill\cdots\otimes (\xi_{t_{l+1}}\otimes e_{k_{l+1}}))$ \\

\item 
$D_{4,i,k}(m)
:= W(e_{k_1})\cdots \widehat{W(e_{k_i})}\cdots W(e_{k_{l+1}}) \otimes W(\xi_{t_1}\otimes e_{k_1})\cdots\\
\null\hfill\cdots \widehat{W(\xi_{t_i}\otimes e_{k_i})}\cdots W(\xi_{t_{l+1}}\otimes e_{k_{l+1}})$.\\

\item 
$D(m)=\displaystyle\sum_{i=2}^{l+1}m^{-\frac{l+1}{2}}\displaystyle\sum_{k_1=1}^{m}\sum_{\substack{k:[l+1]\setminus\{1\}\mapsto[m]\\ injective}}\Bigg( D_{1,i,k}(m)+\inner{\J\xi_{t_1}}{\xi_{t_i}}_Uq^{i-1}D_{2,i,k}(m) \\ \null\hfill +\left(\prod_{2\leqslant j\leqslant i-1}\tilde{q}_{k_i,k_j}\right)D_{3,i,k}(m)\Bigg),$

\end{enumerate}
where for a vector $\eta$, $\widehat{\eta}$ denotes the absence of the term $\eta$ in the expression.\\\\
Our aim will be to show that the sequence  $(D(m))$ represents the    $ 0$-element in the ultraproduct. We need a few lemmas to prove our assertion. 
\begin{lem}\label{Bound}
 Given the following choice of vectors $\xi_1,\ldots,\xi_l \in \h_{\C}$, there exists a constant 
 $D(l) > 0$, such that for any  $m\in \N$ and all $k :[l]\mapsto [m]$, the following inequalities hold:\\

\begin{enumerate}
    \item 

$\norm{W((\xi_1 \otimes e_{k_1})\otimes\cdots\otimes(\xi_{l} \otimes e_{k_l}))} \leqslant D(l),$  \\

\item $\norm{W(\xi_2 \otimes e_{k_1})\cdots\widehat{W(\xi_j\otimes e_{k_j})}\cdots W(\xi_{l} \otimes e_{k_l})}\leqslant D(l),$ \\

\item 
$\norm{W(e_{k_1})\cdots \widehat{W(e_{k_j})}\cdots W( e_{k_l})}\leqslant D(l).$ \\
\end{enumerate}

\end{lem}
\begin{proof}

Recall that for $\xi\in\h_{\C}$, we have $W(\xi)= l(\xi)+l^*(\J\xi)$. Hence, $\norm{W(\xi)}\leqslant\norm{l(\xi)}+\norm{l^*(\J\xi) }\leqslant 2(1-q)^{-1/2}\max\set{\norm{\xi}, \norm{\J\xi}}.$ So 
\begin{align*}
   \|W(\xi_2 \otimes e_{k_2})\cdots&\widehat{W(\xi_j\otimes e_{k_j})}\cdots W(\xi_{l} \otimes e_{k_l})\| \\ &
   \leqslant (\frac{2^{l-1}}{(1-q)^{(l-1)/2}})\max\set{\norm{\xi_2},\norm{\J\xi_2}}\cdots\max\set{\norm{\xi_l},\norm{\J\xi_l}}.
\end{align*}
Note that the right-hand side of the inequality is a constant only depending upon $l$ and the vectors $\xi_1,\ldots,\xi_l \in \h_{\C},$ establishing the second inequality. Similar arguments can be used to obtain the third inequality.

Now, for the first inequality, we use the Khintchine type inequality from \cite[Theorem 3.10]{BKM21} with $K = \C$.  It follows that 
\begin{align*}
    \|W((\xi_1 \otimes e_{k_1}) & \otimes\cdots\otimes(\xi_{l} \otimes e_{k_l}))\| \\ & \leqslant C_q (l+1) \max_{0\leqslant j\leqslant l}\norm{ (1 \otimes S_{\varphi})(R^*_{l,j} ((\xi_1 \otimes e_{k_1} )\otimes\cdots\otimes(\xi_l\otimes e_{k_l} )))}.
\end{align*}

After writing the formula for $R^*_{l,j}$, the term $(1 \otimes S_{\varphi})(R^*_{l,j} ((\xi_1 \otimes e_{k_1} )\otimes\cdots\otimes(\xi_l\otimes e_{k_l} )))$ can be written as sums of simple tensors. We can see that the norms of each term in the sum are bounded by a constant depending only on $l$.

\end{proof}

We fix some vectors $\xi_1,\xi_2,\ldots,\xi_l\in\h_{\C}$ for some $l\in\N$ and let $m\in\N$ be such that $m\geqslant l$. For some $1\leqslant i\leqslant l$, we denote $\Lambda_{i} \subseteq [m]^l$ by the set of ordered pairs
$k=(k_1,k_2,\ldots,k_l)\in [m]^l$ where $k_j$'s, for $1\leqslant j\neq i\leqslant l$, are distinct   except $k_1 = k_i.$ Let $\abs{\Lambda_i}$ denote the number of elements in $\Lambda_i$. 

Now
we use the similar idea as \cite[Proposition 4.5]{ABW18} to prove our next lemma. We remark that in \cite[Proposition 4.5]{ABW18}, it was proved for $q$-Gaussian von Neumann algebras, but we do it for mixed $q$-Gaussian von Neumann algebras.\\

\begin{lem}
   Let $K$ be a Hilbert space  and $(A_{k})_{k\in \Lambda_i} \subset \B(K)$ be any family of operators. For any $ J\in \Pi_{l,p}$, there exist a constant $C(l)$ (dependent only on $l$) such that \\
\begin{equation}\label{eq1}
\norm{\sum_{k\in\Lambda_i}A_{k} \otimes (e_{k_{J^c}}^{\otimes(l-p)}\otimes e_{k_J}^{\otimes p})} \leq C(l) m^{ l/2}\sup_{k\in\Lambda_i}\norm{A_{k}},\\
\end{equation}
where the norm is calculated in $\B(K)\otimes_{min} (\h^{\otimes_T^{l-p}})_c\otimes_{min} (\h^{\otimes_T^p})_r $. Further, it follows that  
 
\begin{equation}\label{eq2}
\norm{\displaystyle\sum_{k \in\Lambda_i}A_{k} \otimes W(e_{k_1}\otimes\cdots\otimes e_{k_l})}\leqslant C(l)m^{l/2}\sup_{k\in\Lambda_i}\norm{A_{k}}. \\
\end{equation}
\end{lem}

\begin{proof}
  We fix a $J\in\Pi_{l,p}$ first. Recall that $J$ is an ordered pair $(j(1),\ldots,j(p))\in[l]^p$ with $1\leqslant j(1)< \cdots <j(p)\leqslant l$ and $J^c=(i(1),\ldots,i(l-p))\in[l]^{l-p}$ with $1\leqslant i(1)< \cdots <i(l-p)\leqslant l$. Then, for each $k = (k_1,k_2,\ldots,k_l)\in\Lambda_i$, $k_J$ denotes the tuple $(k_{j(1)},\ldots,k_{j(p)})$ and $k_{J^c}=(k_{i(1)},\ldots,k_{i(l-p)})$. Thus, the tensor $e_{k_{J^c}}^{\otimes(l-p)}\otimes e_{k_J}^{\otimes p}$ is of the form $e_{k_{i(1)}}\otimes\cdots\otimes e_{k_{i(l-p)}} \otimes e_{k_{j(1)}}\otimes\cdots\otimes e_{k_{j(p)}}$. Clearly, the collection of vectors $\set{e_{k_{J^c}}^{\otimes(l-p)}:\ k\in\Lambda_i}$ forms an orthonormal set in $\h^{\otimes{l-p}}$ and similarly, $\set{e_{k_J}^{\otimes p}:\ k\in\Lambda_i}$ is an orthonormal set in  $\h^{\otimes{p}}$. In fact, this is true for any $J\in\Pi_{l,p}.$ Since, for given $k=(k_1,\ldots,k_l)$, application of $J$ gives the tuples $k_J=(k_{i(1)},\ldots,k_{i(l-p)})$ and $k_{J^c}=(k_{j(1)},\ldots,k_{j(p)})$.  Therefore, we can identify $\set{e_{k_{J^c}}^{\otimes(l-p)}:\ k\in\Lambda_i}$ with a different orthonormal set $(v_a)_{a\in A}\subseteq \h^{\otimes{l-p}}$ and  $\set{e_{k_{J}}^{\otimes p}:\ k\in\Lambda_i}$ with a orthonormal set $(w_b)_{b\in B}\subseteq \h^{\otimes p}$. This implies  $ e_{k_{J^c}}^{\otimes(l-p)}\otimes e_{k_J}^{\otimes p}=v_a\otimes w_b$, for some $a\in A$ and $b\in B$.  
  
Note that we are dealing with  deformed  Hilbert spaces 
$\h^{\otimes_T^{l-p}}$ and $\h^{\otimes_T^ p}$ that are deformations of the spaces $\h^{\otimes{l-p}}$ and $\h^{\otimes{p}}$ given by the operators $P_T^{l-p}$ and $P_T^{p}$,
respectively. So, let $\eta_{k}$ be the vector in $\h^{\otimes_T^{l-p}}\otimes\h^{\otimes_T^p} $ such that $v_a\otimes w_b = (P_T^{l-p})^{1/2}\otimes (P_T^{p})^{1/2}(\eta_{k} )$.  Let $(v^{\prime}_a)_{a\in A}\subseteq \h^{\otimes_T^{l-p}}$ and $(w^{\prime}_b)_{b\in B}\subseteq \h^{\otimes_T^p}$ be orthonormal subsets and write $\eta_{k} = v_a^{\prime}\otimes w_b^{\prime}$ for some $a\in A$ and $b\in B$. Since both the row and column Hilbert spaces are homogeneous operator spaces, and the minimal tensor product of cb maps is cb again, which follows as a consequence of \cite[Eq 2.1.3]{P03},  we have the following bound:
\begin{align*}
    \max_{0\leqslant p\leqslant l}\norm{\sum_{k\in\Lambda_i}A_{k} \otimes (e_{k_{J^c}}^{\otimes(l-p)}\otimes e_{k_J}^{\otimes p})}\leqslant C(l)\max_{0\leqslant p\leqslant l}\norm{\sum_{k\in\Lambda_i}A_{k} \otimes \eta_{k}},
\end{align*} 
where $C(l)$ is coming from the bounds of the operators $(P_T^{l-p})^{1/2}$ and $(P_T^{p})^{1/2}$.

We have  $(\h^{\otimes_T^{ l-p}})_{c}\otimes_{min} (\h^{\otimes_T^p})_r\simeq(\h^{\otimes_T^{ l-p}})_{c}\otimes_h (\h^{\otimes_T^p})_r\simeq \mathcal{K}(\h^{\otimes_T^p},\h^{\otimes_T^{ l-p}})$, which follows from Remark \ref{isometric}. Using these complete isomorphism in Remark \ref{isometric}(2), we can view the vector $\eta_{k} $ as a matrix units in $\mathcal{K}(\h^{\otimes_T^p},\h^{\otimes_T^{ l-p}})$. So the term $A_{k} \otimes\eta_{k}$ can be realized as a matrix with all its entries as $A_{k}$. Using the relation between the operator norm and the Hilbert-Schmidt norm, we get
\begin{align*}
 \norm{\sum_{\Lambda_i}A_{k}\otimes (e_{k_1}\otimes\cdots\otimes e_{k_l}) }\leqslant C(l)\left(\sum_{\Lambda_i}\norm{A_{k}}^2\right)^{1/2}&\leqslant C(l)\left(|\Lambda_i|\sup_{k\in\Lambda_i}\norm{A_{k}}^2\right)^{1/2}
\end{align*}
Recall, that $\Lambda_{i}$ is the set of ordered pairs $(k_1,k_2,\ldots,k_l)\in [m]^l$, where $k_j$’s are distinct  except $k_1 = k_i.$ Hence, it follows that $\abs{\Lambda_i}\leqslant m^l$, which in turn gives the following bound.
 \begin{align*}
    \norm{\sum_{\Lambda_i}A_{k}\otimes (e_{k_1}\otimes\cdots\otimes e_{k_l}) }\leqslant C(l)\left(m^l\sup_{k\in\Lambda_i}\norm{A_{k}}^2\right)^{1/2}=C(l) m^{l/2}\sup_{k\in\Lambda_i}\norm{A_{k}},
 \end{align*}
 which is the inequality $(\ref{eq1})$.

 % So we rewrite the sum $\sum_{\lambda\in\Lambda_i}A_{\lambda} \otimes (e_{k_{J^c}}^{\otimes(l-p)}\otimes e_{k_J}^{\otimes p})$ as $\sum_{\lambda\in\Lambda_i}A_{\lambda} \otimes e_{\lambda_J}$. We consider different orthonormal sets $(v_k)_{k\in \Lambda_i}\subseteq \h^{\otimes{l-p}}$ and $(w_b)_{k\in \Lambda_i}\subseteq \h^{\otimes p}$.\\

Now, for the second part of the lemma, notice that $W(e_{k_1}\otimes\cdots\otimes e_{k_l})$'s are elements of the mixed $q$-Gaussian factor $\Gamma_{\tilde{Q}}(\R^m)^{\prime\prime}$. Consider the element $\sum_{k\in\Lambda_i}A_{k} \otimes (e_{k_1}\otimes\cdots\otimes e_{k_l})\in \B(K)\otimes_{min}\h^{\otimes_T ^l}$, where $\h=\C^{m}$. It follows from \cite[Theorem 3.10]{BKM21} that
\begin{align*}
&\norm{\displaystyle\sum_{k\in\Lambda_i}A_{k} \otimes W(e_{k_1}\otimes\cdots\otimes e_{k_l})}\\ & \qquad\qquad\qquad\qquad \leqslant  C_q(l+1)
\max_{0\leqslant p\leqslant l}\norm{\displaystyle\sum_{k\in\Lambda_i}A_{k} \otimes R^*_{l,m}(e_{k_1}\otimes\cdots\otimes e_{k_l})},
\end{align*}
where the norm in the right-hand side is taken in $\B(K)\otimes_{min} (\h^{\otimes_T^{l-p}})_c\otimes_{min} (\h^{\otimes_T^p})_r $. Hence, in order to establish the inequality $(18)$, it is enough to get an estimate for the term $\max_{0\leqslant p\leqslant l}\norm{\sum_{k\in\Lambda_i}A_{k} \otimes R^*_{l,p}(e_{k_1}\otimes\cdots\otimes e_{k_l})}.$  Recall from Eq.\ref{Op R} that
\begin{align*}
R^*_{l,p}(e_{k_1}\otimes\cdots\otimes e_{k_l})= \displaystyle\sum_{J\in \Pi_{l,p}} f_{(J^c,J)}(\tilde{q}_{i,j} )(e_{k_{J^c}}^{\otimes(l-p)}\otimes e_{k_J}^{\otimes p}). 
\end{align*}
 In fact, the operator  $R^*_{l,p}$ viewed as an operator from $\h^{\otimes_T^{l}} $ to $\h^{\otimes_T^{l-p}}\otimes \h^{\otimes_T^p}$ is bounded, which follows from \cite[Lemma 3.6(i)]{BKM21}. Note that $R^*_{l,p}(e_{k_1}\otimes\cdots\otimes e_{k_l})$ can be realized as an element of $ (\h^{\otimes_T^{l-p}})_c\otimes_{min} (\h^{\otimes_T^p})_r$. This implies  
\begin{align*}
    &\norm{\sum_{k\in\Lambda_i}A_{k} \otimes R^*_{l,p}(e_{k_1}\otimes\cdots\otimes e_{k_l})}\\ &\qquad\qquad\qquad\qquad= \norm{\sum_{k\in\Lambda_i}A_{k} \otimes \displaystyle\sum_{J\in \Pi_{l,p}} f_{(J^c,J)}(\tilde{q}_{i,j} )(e_{k_{J^c}}^{\otimes(l-p)}\otimes e_{k_J}^{\otimes p})}\\
    &\qquad\qquad\qquad\qquad\leqslant \displaystyle\sum_{J\in \Pi_{l,p}} \abs{f_{(J^c,J)}(\tilde{q}_{i,j} )}\norm{\sum_{k\in\Lambda_i}A_{k} \otimes (e_{k_{J^c}}^{\otimes(l-p)}\otimes e_{k_J}^{\otimes p})}
\end{align*}
However, from the first part, we have an estimate for $\norm{\sum_{k\in\Lambda_i}A_{k} \otimes (e_{k_{J^c}}^{\otimes(l-p)}\otimes e_{k_J}^{\otimes p})},$ for any $ J\in \Pi_{l,p}$. Recall that the bound is independent of $J$. Hence, 
\begin{align*}
    \norm{\sum_{k\in\Lambda_i}A_{k} \otimes R^*_{l,p}(e_{k_1}\otimes\cdots\otimes e_{k_l})}\leqslant m^{1/2}C(l)\sup_{k\in\Lambda_i}\norm{A_{k}}\displaystyle\sum_{J\in \Pi_{l,p}} \abs{f_{(J^c,J)}(\tilde{q}_{i,j} )}.
\end{align*}
Clearly, $\sum_{J\in \Pi_{l,p}} \abs{f_{(J^c,J)}(\tilde{q}_{i,j} )}$ is summable and the sum is a constant only dependent on $l$. 
Multiplying the sum with $C(l)$, we get a constant dependent only on $l$, which we write as $C(l)$ again, with a slight abuse of notation. Therefore, we have
\begin{align*}
    \norm{\sum_{k\in\Lambda_i}A_{k} \otimes R^*_{l,p}(e_{k_1}\otimes\cdots\otimes e_{k_l})}\leqslant m^{1/2}C(l)\sup_{k\in\Lambda_i}\norm{A_{k}}.
\end{align*}
\end{proof}

% $$\textcolor{red}{************************}$$

We write the following lemma for reference, although the proof follows from \cite[Proposition 4.5]{ABW18}.
\begin{lem}
Let $K$ be a Hilbert space  and $(A_{k})_{k\in \Lambda_i} \subset \B(K)$ be any family of operators. Then, we have the following inequalities.

\begin{equation}\label{eq3}
\norm{\displaystyle\sum_{k\in\Lambda_i}A_{k} \otimes W((\xi_1\otimes e_{k_1})\otimes\cdots\otimes (\xi_l\otimes e_{k_l}))}\leqslant C(l)\sup_{k\in\Lambda_i}\norm{A_{k}}m^{l/2},  
\end{equation}

where $C(l) > 0$ depends only on $d$ and the choice of vectors $\xi_1,\xi_2,\ldots,\xi_l \in\h_{\C}$.
    
\end{lem}
\begin{proof}
 Recall that $W((\xi_1\otimes e_{k_1})\otimes\cdots\otimes (\xi_l\otimes e_{k_l})$'s are elements of $ \Gamma_q(\h_{\R}\otimes \R^m, U_{t}\otimes I_m)^{\prime\prime},$ which is $q$-Araki Woods factor. Hence, the inequality will follow from the second inequality of \cite[Proposition 4.5]{ABW18}.
 \end{proof}

\begin{lem}\label{D}
 The sequence $(D(m))_m$ represents the $0$-element in the ultraproduct $\mathcal{N}$.   
\end{lem} 
\begin{proof}
First, recall that 
\begin{align*}
D(m)=&\displaystyle\sum_{i=2}^{l+1}m^{-\frac{l+1}{2}}\displaystyle\sum_{k_1=1}^{m}\sum_{\substack{k:[l+1]\setminus\{1\}\mapsto[m]\\ injective}}\Bigg( D_{1,i,k}(m)+\inner{\J\xi_{t_1}}{\xi_{t_i}}_Uq^{i-1}D_{2,i,k}(m) \\ &\qquad\qquad\qquad\qquad\qquad\qquad\qquad +\left(\prod_{2\leqslant j\leqslant i-1}\tilde{q}_{k_i,k_j}\right)D_{3,i,k}(m)\Bigg).
\end{align*}
Now since, $l\in\N$ and $q$, $\tilde{q}_{k_i,k_j}$'s are fixed in the expression of $D(m)$ and $D(m)$ is the linear combination of $D_{1,i, k}(m), D_{2,i, k}(m),$ and $D_{3,i, k}(m) $, thus,  to show that the sequence $(D(m))_m$ represents the $0$-element in the ultraproduct $\mathcal{N}$, it is enough to show  
\begin{align*}
\lim_{m\rightarrow\infty}m^{\frac{-(l+1)}{2}}\norm{D_{j,i,k}}=0\ \ \text{for all} \ 1\leqslant j\leqslant 3.    
\end{align*}
First notice  that, $D_{1,i, k}(m)$ is of the form $$\sum_{\Lambda_i} W((e_{k_1})\otimes\cdots\otimes ( e_{k_{l+1}}))\otimes A_{k} ,$$ where $A_{k}$ stands for $W((\xi_{t_1}\otimes e_{k_1})\otimes\cdots\otimes (\xi_{t_{l+1}}\otimes e_{k_{l+1}}))$. Then, from the inequality (\ref{eq1}) of  the last lemma, along with the commutativity of the tensor product, we get
\begin{align*}
    \lim_{m\rightarrow\infty}m^{\frac{-(l+1)}{2}}\norm{D_{1,i,k}} & =\lim_{m\rightarrow\infty}m^{\frac{-(l+1)}{2}}\norm{\sum_{\Lambda_i}A_{k}\otimes  W((e_{k_1})\otimes\cdots\otimes ( e_{k_{l+1}}))}\\
    & \leqslant\lim_{m\rightarrow\infty}m^{\frac{-1}{2}}C(l)\sup_{k\in\Lambda_i}\norm{A_{k}}\\
    & \leqslant\lim_{m\rightarrow\infty}m^{\frac{-1}{2}}C(l)D(l)\qquad \text{(from Lemma \ref{Bound}(1))}\\
    & = 0. 
\end{align*}
Similarly, by using the inequality (\ref{eq1}), Lemma \ref{Bound}(2), and a careful application of commutativity of tensors, we obtain 
\begin{align*}
    \lim_{m\rightarrow\infty}m^{\frac{-(l+1)}{2}}\norm{D_{2,i,k}}& = \lim_{m\rightarrow\infty}m^{\frac{-(l+1)}{2}}\norm{\sum_{\Lambda_i}A_{k}\otimes W(e_{k_1} \otimes\cdots\otimes e_{k_{l+1}})}\\
    & \leqslant \lim_{m\rightarrow\infty}m^{\frac{-1}{2}}C(l)\sup_{k\in\Lambda_{i}}\norm{A_{k}}\\
    & \leqslant \lim_{m\rightarrow\infty}m^{\frac{-1}{2}}C(l)D(l)\\
    & =0,
\end{align*}
where $A_{k}$ stands for $W(\xi_{t_1}\otimes e_{k_1})\cdots\widehat{W(\xi_{t_i}\otimes e_{k_i})}\cdots W(\xi_{t_{l+1}}\otimes e_{k_{l+1}}).$

Also, for $A_{k}= W(e_{k_1})\cdots \widehat{W(e_{k_i})}\cdots W(e_{k_{l+1}})$ in the inequality (\ref{eq3}) along with Lemma \ref{Bound}(3) implies
\begin{align*}
&\lim_{m\rightarrow\infty}m^{\frac{-(l+1)}{2}}\norm{D_{3,i,k}}\\
&\qquad\qquad =\lim_{m\rightarrow\infty}m^{\frac{-(l+1)}{2}}\norm{\sum_{\Lambda_i}A_{k}\otimes W((\xi_{t_1}\otimes e_{k_1})\otimes\cdots\otimes (\xi_{t_{l+1}}\otimes e_{k_{l+1}})}\\
& \qquad\qquad \leqslant \lim_{m\rightarrow\infty}m^{\frac{-1}{2}}C(l)\sup_{k\in\Lambda_{i}}\norm{A_{k}}\\
& \qquad\qquad \leqslant \lim_{m\rightarrow\infty}m^{\frac{-1}{2}}C(l)D(l)\\
& \qquad\qquad = 0.    
\end{align*}
This suggests that $(D(m))_{\omega} = 0$.\\

\end{proof}

\begin{proof}\emph{of Theorem \ref{image}}:
We will prove the theorem by induction on $l$. For $l=1$, it is clear from the definition of $\Theta$ in Theorem \ref{Embedding}. Suppose it is true for all $1\leqslant l'\leqslant l$. We will prove for the case of $l+1$. Let $\xi_{t_1},\ldots,\xi_{t_{l+1}}\in\h_{\C}$. Then, it can be easily checked that 
\begin{align*}
W(\xi_{t_1}\otimes\cdots\otimes\xi_{t_{l+1}})= & W(\xi_{t_1})W(\xi_{t_2}\otimes \cdots\otimes\xi_{t_{l+1}})\\ &-\displaystyle\sum_{i=2}^{l+1}\inner{\J\xi_{t_1}}{\xi_{t_i}}_U\left(\displaystyle\prod_{2\leqslant j\leqslant i-1}q_{t_i,t_j}\right)W(\xi_{t_2}\otimes\cdots\otimes\widehat{\xi}_{t_i}\cdots\otimes\xi_{t_{l+1}}),
\end{align*}
By applying $\Theta$ on both sides and using the induction hypothesis, we get, 
\begin{align}\label{image decomp}
&\Theta(W(\xi_{t_1}\otimes\cdots\otimes\xi_{t_{l+1}}))  \\ 
&\qquad =W^{\omega}(\xi_{t_1})W^{\omega}(\xi_{t_2}\otimes \cdots\otimes\xi_{t_{l+1}}) \nonumber \\ 
&\qquad\qquad\qquad\qquad -\displaystyle\sum_{i=2}^{l+1}\inner{\J\xi_{t_1}}{\xi_{t_i}}_U\left(\displaystyle\prod_{2\leqslant j\leqslant i-1}q_{t_i,t_j}\right)W^{\omega}(\xi_{t_2}\otimes\cdots\otimes\widehat{\xi}_{t_i}\cdots\otimes\xi_{t_{l+1}}). \nonumber
\end{align}
Let us consider the first term on the right-hand side. 
{\scriptsize
\begin{align*}
&W^{\omega}(\xi_{t_1})W^{\omega}(\xi_{t_2}\otimes \cdots\otimes\xi_{t_{l+1}})\\
&=\left(m^{-\frac{1}{2}}\displaystyle\sum_{k_1=1}^m W(e_{k_1})\otimes W(\xi_{t_1}\otimes e_{k_1})\right)_{\omega}\times \\
&\qquad\qquad \left(m^{-\frac{l}{2}}\displaystyle\sum_{\substack{k:[l+1]\setminus\{1\}\mapsto[m]\\ injective }} W(e_{k_2})\cdots W(e_{k_{l+1}})\otimes W(\xi_{t_2}\otimes e_{k_2})\cdots W(\xi_{t_{l+1}}\otimes e_{k_{l+1}})\right)_{\omega}\\
&=\left(m^{-\frac{l+1}{2}}\displaystyle\sum_{k_1=1}^{m}\sum_{\substack{k:[l+1]\setminus\{1\}\mapsto[m]\\ injective}} W(e_{k_1})\cdots W(e_{k_{l+1}})\otimes W(\xi_{t_1}\otimes e_{k_1})\cdots W(\xi_{t_{l+1}}\otimes e_{k_{l+1}})\right)_{\omega}\\
&=\left(m^{-\frac{l+1}{2}}\displaystyle\sum_{\substack{k:[l+1]\mapsto[m]\\ injective}} W(e_{k_1})\cdots W(e_{k_{l+1}})\otimes W(\xi_{t_1}\otimes e_{k_1})\cdots W(\xi_{t_{l+1}}\otimes e_{k_{l+1}})\right)_{\omega}\\
&+\left(m^{-\frac{l+1}{2}}\displaystyle\sum_{k_1=1}^{m}\sum_{i=2}^{l+1}\sum_{\substack{k:[l+1]\setminus\{1\}\mapsto[m]\\ injective\\ k_1=k_i}} W(e_{k_1})\cdots W(e_{k_{l+1}})\otimes W(\xi_{t_1}\otimes e_{k_1})\cdots W(\xi_{t_{l+1}}\otimes e_{k_{l+1}})\right)_{\omega}
\end{align*} 
}

Recall that the first term of the sum in the last equality is nothing but $W^{\omega}(\xi_{t_1}\otimes\cdots\otimes\xi_{t_{l+1}}).$ So, we have
{\scriptsize
\begin{align*}
&W^{\omega}(\xi_{t_1})W^{\omega}(\xi_{t_2}\otimes \cdots\otimes\xi_{t_{l+1}})\\
&=W^{\omega}(\xi_{t_1}\otimes\cdots\otimes\xi_{t_{l+1}})\\
&+\displaystyle\sum_{i=2}^{l+1}\left(m^{-\frac{l+1}{2}}\displaystyle\sum_{k_1=1}^{m}\sum_{\substack{k:[l+1]\setminus\{1\}\mapsto[m]\\ injective\\ k_1=k_i}}W(e_{k_1})\cdots W(e_{k_{l+1}})\otimes W(\xi_{t_1}\otimes e_{k_1})\cdots W(\xi_{t_{l+1}}\otimes e_{k_{l+1}})\right)_{\omega}.
\end{align*}
}
Since the first term $W^{\omega}(\xi_{t_1}\otimes\cdots\otimes\xi_{t_{l+1}})$ is what we need to prove our result, we have to find a way to get rid of the second term in the sum of the above equation. For $k_1=k_i$ and $k_2\neq\cdots\neq k_l$,
\begin{align*}
W(e_{k_1})\cdots W(e_{k_{l+1}})&=W(e_{k_1}\otimes\cdots\otimes e_{k_{l+1}})\\ &  +\left(\displaystyle\prod_{1\leqslant j\leqslant i-1}\tilde{q}_{k_i,k_j}\right)W(e_{k_1})\cdots \widehat{W(e_{k_i})}\cdots W(e_{k_{l+1}})
\end{align*}
and
\begin{align*}
&W(\xi_{t_1}\otimes e_{k_1})\cdots W(\xi_{t_{l+1}}\otimes e_{k_{l+1}})\\ &=W((\xi_{t_1}\otimes e_{k_1})\otimes\cdots\otimes (\xi_{t_{l+1}}\otimes e_{k_{l+1}}))\\
& \qquad\qquad +\inner{\J\xi_{t_1}}{\xi_{t_i}}_Uq^{i-1}W(\xi_{t_1}\otimes e_{k_1})\cdots \widehat{W(\xi_{t_i}\otimes e_{k_i})}\cdots W(\xi_{t_{l+1}}\otimes e_{k_{l+1}}).
\end{align*}
So, if we take the tensor product of the two terms $W(e_{k_1})\cdots W(e_{k_{l+1}})$ and $ W(\xi_{t_1}\otimes e_{k_1})\cdots W(\xi_{t_{l+1}}\otimes e_{k_{l+1}})$, there will be four terms involving $D_{1,i,k}$, $ D_{2,i,k}$, $ D_{3,i,k}$, and $D_{4,i,k}$. Observe that 
\begin{align*}
&W(e_{k_1})\cdots W(e_{k_{l+1}})\otimes W(\xi_{t_1}\otimes e_{k_1})\cdots W(\xi_{t_{l+1}}\otimes e_{k_{l+1}})\\
&=D_{1,i,k}(m)+\inner{\J\xi_{t_1}}{\xi_{t_i}}_Uq^{i-1}D_{2,i,k}(m)\\
&  +\left(\displaystyle\prod_{1\leqslant j\leqslant i-1}\tilde{q}_{k_i,k_j}\right)D_{3,i,k}(m)+ \left(\displaystyle\prod_{1\leqslant j\leqslant i-1}\tilde{q}_{k_i,k_j} \right)q^{i-1} \inner{\J\xi_{t_1}}{\xi_{t_i}}_UD_{4,i,k}(m)\\
&=D_{1,i,k}(m)+\inner{\J\xi_{t_1}}{\xi_{t_i}}_Uq^{i-1}D_{2,i,k}(m)\\
&  \qquad  +\left(\displaystyle\prod_{1\leqslant j\leqslant i-1}\tilde{q}_{k_i,k_j}\right)D_{3,i,k}(m)+ \left(\displaystyle\prod_{1\leqslant j\leqslant i-1}q_{k_i,k_j} \right)\inner{\J\xi_{t_1}}{\xi_{t_i}}_UD_{4,i,k}(m).
\end{align*}
% where
% \begin{align*}
% D_{1,i}(m):=W(e_{k_1}&\otimes\cdots\otimes e_{k_{l+1}})\otimes W((\xi_{t_1}\otimes e_{k_1})\otimes\cdots\otimes (\xi_{t_{l+1}}\otimes e_{k_{l+1}})),
% \end{align*}
% \begin{align*}
% D_{2,i}(m):= W(e_{k_1}&\otimes\cdots\otimes e_{k_{l+1}})\otimes W(\xi_{t_1}\otimes e_{k_1})\cdots\\
% &\qquad\qquad\qquad\qquad\qquad\cdots \widehat{W(\xi_{t_i}\otimes e_{k_i})}\cdots W(\xi_{t_{l+1}}\otimes e_{k_{l+1}}),
% \end{align*}
% \begin{align*}
% D_{3,i}(m):= W(e_{k_1})\cdots \widehat{W(e_{k_i})}\cdots W(e_{k_{l+1}}) \otimes W((\xi_{t_1}\otimes e_{k_1})\otimes\cdots\otimes (\xi_{t_{l+1}}\otimes e_{k_{l+1}}),
% \end{align*}
% and 
% \begin{align*}
% &D_{4,i}(m)\\
% &:= W(e_{k_1})\cdots \widehat{W(e_{k_i})}\cdots W(e_{k_{l+1}}) \otimes W(\xi_{t_1}\otimes e_{k_1})\cdots\\
% &\qquad\qquad\qquad\qquad\qquad\qquad\qquad\qquad\qquad\cdots \widehat{W(\xi_{t_i}\otimes e_{k_i})}\cdots W(\xi_{t_{l+1}}\otimes e_{k_{l+1}}).
% \end{align*}
Here the coefficients $\prod_{1\leqslant j\leqslant i-1}q_{k_i,k_j}$ of $D_{4,i,k}(m)$ is because of the fact that $q_{i,j}=q\tilde{q}_{i,j}$. Notice that, 
\begin{align*}
\displaystyle\sum_{i=2}^{l+1}m^{-\frac{l+1}{2}}&\displaystyle\sum_{k_1=1}^{m}\sum_{\substack{k:[l+1]\setminus\{i\}\mapsto[m]\\ injective\\ k_1=k_i}} \left(\displaystyle\prod_{2\leqslant j\leqslant i-1}q_{k_i,k_j}\right) \inner{\J\xi_{t_1}}{\xi_{t_i}}_UD_{4,i,k}(m)\\ 
&= \displaystyle\sum_{i=2}^{l+1}m^{-\frac{l-1}{2}}\displaystyle\sum_{\substack{k:[l+1]\setminus\{1,i\}\mapsto[m]\\ injective}} \left(\displaystyle\prod_{2\leqslant j\leqslant i-1}q_{k_i,k_j}\right) \inner{\J\xi_{t_1}}{\xi_{t_i}}_U D_{4,i,k}(m).
\end{align*}
This is because terms involving $k_i$ and $k_1$ are absent in the summation. So, we can remove the condition $k_i=k_1$ and take the sum over $k_1$. Now, the second summation is over $l-1$ indices. Using the definition of $D_{4,i,k}(m)$, we can see that the sequence $\left(m^{-\frac{l-1}{2}}\sum_{\substack{k:[l+1]\setminus\set{1,i}\mapsto[m]\\ injective}}D_{4,i,k}(m)\right)_m$ is a representative sequence of the following element of the ultraproduct $\mathcal{N}$:
\begin{align*}
\displaystyle\sum_{i=2}^{l+1}\left(\displaystyle\prod_{2\leqslant j\leqslant i-1}q_{k_i,k_j}\right) \inner{\J\xi_{t_1}}{\xi_{t_i}}_UW^{\omega}(\xi_{t_1}\otimes\cdots\otimes\widehat{ \xi}_{t_i} \otimes\cdots\otimes \xi_{t_{l+1}}).
\end{align*}
Therefore, we can write, 
\begin{align}\label{wick prod decomp}
& W^{\omega}(\xi_{t_1})W^{\omega}(\xi_{t_2}\otimes \cdots\otimes\xi_{t_{l+1}})\\
& \nonumber =W^{\omega}(\xi_{t_1}\otimes\cdots\otimes\xi_{t_{l+1}})\\
& \nonumber +\displaystyle\sum_{i=2}^{l+1}\left(\displaystyle\prod_{2\leqslant j\leqslant i-1}q_{k_i,k_j}\right) \inner{\J\xi_{t_1}}{\xi_{t_i}}_UW^{\omega}(\xi_{t_1}\otimes\cdots\otimes\widehat{ \xi}_{t_i} \otimes\cdots\otimes \xi_{t_{l+1}})+(D(m))_{\omega},
\end{align}

where
\begin{align*}
D(m)=&\displaystyle\sum_{i=2}^{l+1}m^{-\frac{l+1}{2}}\displaystyle\sum_{k_1=1}^{m}\sum_{\substack{k:[l+1]\setminus\{1\}\mapsto[m]\\ injective\\ k_1=k_i}}\Bigg( D_{1,i,k}(m)+\inner{\J\xi_{t_1}}{\xi_{t_i}}_Uq^{i-1}D_{2,i,k}(m) \\
&\qquad\qquad \qquad \qquad \qquad \qquad \qquad \qquad  +\left(\prod_{2\leqslant j\leqslant i-1}\tilde{q}_{k_i,k_j}\right)D_{3,i,k}(m)\Bigg).
\end{align*}
It follows from Lemma \ref{D} that the $(D(m))_{\omega}$ is $`0$' in $\mathcal{N}$. Hence, from Eq.\ref{wick prod decomp} and Eq.\ref{image decomp} it follows that  
\begin{align*}
\Theta(W(\xi_{t_1}\otimes\cdots\otimes\xi_{t_{l+1}}))=W^{\omega}(\xi_{t_1}\otimes\cdots\otimes\xi_{t_{l+1}}),
\end{align*}
which proves the theorem.
\end{proof}

We require the following intertwining result for $F_n$ using the the isomorphism $\Theta$ from Theorem \ref{Embedding}, which will facilitate the shifting of cb norm of the radial multipliers from the mixed $q$-Gaussian von Neumann factor to the mixed $q$-deformed Araki-Woods factor.

\begin{thm}\label{intertwin}
Let $F_n: \Gamma_T(\h_{\R},U_t)^{\prime\prime} \mapsto \Gamma_T(\h_{\R},U_t)^{\prime\prime}$ be the projection onto the ultraweakly closed span of $\set{W(\xi) : \xi\in\h^{\otimes n}}$. Then, we have
\begin{equation}
\Theta\circ F_n =p(F_n \otimes Id)_{\omega}p\circ\Theta.
\end{equation}
\end{thm}
\begin{proof}
For $\xi\in\h_{\C}^{\otimes l}$, from Theorem \ref{image}, we have $\Theta(F_n(W(\xi)))=\Theta(\delta_{n,l}W(\xi))=\delta_{n,l}W^{\omega}(\xi).$ Also, 
\begin{align*}
&(F_n\otimes Id )_{\omega}(\Theta(W(\xi)))\\
&=(F_n\otimes Id )_{\omega}(W^{\omega}(\xi))\\
&=(F_n\otimes Id)_{\omega}\left(\left(m^{\frac{-l}{2}}\displaystyle\sum_{\substack{k:[l]\mapsto [m]\\ injective}}^m W(e_{k_1})\cdots W(e_{k_l})\otimes W(\xi_{t_1}\otimes e_{k_1})\cdots W(\xi_{t_l}\otimes e_{k_l})\right)_{\omega} \right)\\
&=\left((F_n\otimes Id)\left(m^{\frac{-l}{2}}\displaystyle\sum_{\substack{k:[l]\mapsto [m]\\ injective}}^m W(e_{k_1})\cdots W(e_{k_l})\otimes W(\xi_{t_1}\otimes e_{k_1})\cdots W(\xi_{t_l}\otimes e_{k_l})\right)\right)_{\omega}\\
&=\left(m^{\frac{-l}{2}}\displaystyle\sum_{\substack{k:[l]\mapsto [m]\\ injective}}^m F_n\left(W(e_{k_1})\cdots W(e_{k_l})\right)\otimes W(\xi_{t_1}\otimes e_{k_1})\cdots W(\xi_{t_l}\otimes e_{k_l})\right)_{\omega}\\
&=\left(m^{\frac{-l}{2}}\displaystyle\sum_{\substack{k:[l]\mapsto [m]\\ injective}}^m F_n\left(W(e_{k_1}\otimes\cdots\otimes e_{k_l})\right)\otimes W(\xi_{t_1}\otimes e_{k_1})\cdots W(\xi_{t_l}\otimes e_{k_l})\right)_{\omega}\\
&=\left(m^{\frac{-l}{2}}\displaystyle\sum_{\substack{k:[l]\mapsto [m]\\ injective}}^m \delta_{n,l}\left(W(e_{k_1}\otimes\cdots\otimes e_{k_l})\right)\otimes W(\xi_{t_1}\otimes e_{k_1})\cdots W(\xi_{t_l}\otimes e_{k_l})\right)_{\omega}\\
&=\delta_{n,l}W^{\omega}(\xi).
\end{align*}
By linearity, this implies that $\Theta\circ F_n =(F_n\otimes Id )_{\omega}\circ\Theta$ on the algebra of Wick words $\widetilde{\Gamma}_T(\h_{\R},U_t)$. Further, cutting down these maps by the support projection $p$, we have  
\begin{align*}
\Theta\circ F_n =p(F_n \otimes Id)_{\omega}p\circ\Theta \qquad \text{on}\qquad \widetilde{\Gamma}_T(\h_{\R},U_t).
\end{align*}
Since, each of these maps are normal and $\widetilde{\Gamma}_T(\h_{\R},U_t)$ ultraweakly dense in $\Gamma_T(\h_{\R},U_t)^{\prime\prime}$, the above equality holds on $\Gamma_T(\h_{\R},U_t)^{\prime\prime}$ as well.
\end{proof}

\begin{thm}
 Let $\phi : \N \mapsto \C$ be a function such that the associated radial multipliers $m_{\phi} : \Gamma_{\tilde{Q}}(\R^m)^{\prime\prime}\mapsto \Gamma_{\tilde{Q}}(\R^m)^{\prime\prime}$ have completely bounded norms uniformly bounded in $m$. Then, the radial multiplier defined by $\phi$ on any mixed $q$-deformed Araki–Woods factor $\Gamma_{T}(\h_{\R},U_t)^{\prime\prime}$ is completely bounded and
\begin{align*}
\norm{m_{\phi} : \Gamma_{T}(\h_{\R},U_t)^{\prime\prime} \mapsto \Gamma_{T}(\h_{\R},U_t)^{\prime\prime}}_{cb} \leqslant \sup_{m\in\N} \norm{m_{\phi} : \Gamma_{\tilde{Q}}(\R^m)^{\prime\prime} \mapsto \Gamma_{\tilde{Q}}(\R^m)^{\prime\prime}}_{cb} 
\end{align*}
\end{thm}
\begin{proof}
From Theorem \ref{intertwin}, we have $\Theta\circ F_n(x) =p(F_n \otimes Id)_{\omega}p\circ\Theta(x)$ for all $x\in\widetilde{\Gamma}_T(\h_{\R},U_t)$. This implies, $\Theta\circ m_{\phi}(x) =p(m_{\phi} \otimes Id)_{\omega}p\circ\Theta(x)$ for all $x\in\widetilde{\Gamma}_T(\h_{\R},U_t)$. Since $\Theta$ is an $*$-isomorphism, we have for all $x\in\widetilde{\Gamma}_T(\h_{\R},U_t)$, $ m_{\phi}(x) =\Theta^{-1}\circ p(m_{\phi} \otimes Id)_{\omega}p\circ\Theta(x)$. 
This implies $m_{\phi}$ can be extended to a completely bounded map on the norm closure of $\widetilde{\Gamma}_T(\h_{\R}, U_t)$, i.e. $m_{\phi}$ is completely bounded on $\Gamma_T(\h_{\R}, U_t)$. Then, by  \cite[Lemma 3.4]{HR11}, $m_{\phi}$ extends to a normal completely bounded map with the same cb norm. Therefore, we have
\begin{align*}
\norm{m_{\phi} : \Gamma_{T}(\h_{\R},U_t)^{\prime\prime} \rightarrow \Gamma_{T}(\h_{\R},U_t)^{\prime\prime}}_{cb}
&=\norm{\Theta^{-1}\circ p(m_{\phi} \otimes Id)_{\omega}p\circ\Theta}_{cb}\\ 
&\leqslant \norm{\Theta^{-1}}_{cb}\norm{p(m_{\phi} \otimes Id)_{\omega}p }_{cb}\norm{\Theta}_{cb}\\
&=\norm{p(m_{\phi} \otimes Id)_{\omega}p }_{cb}\\
&=\norm{{(m_{\phi} \otimes Id)_{\omega}}}_{cb}\\
&\leqslant \sup_{m\in\N} \norm{m_{\phi} : \Gamma_{\tilde{Q}}(\R^m)^{\prime\prime} \rightarrow \Gamma_{\tilde{Q}}(\R^m)^{\prime\prime}}_{cb} 
\end{align*}
Since $\Gamma_{\tilde{Q}}(\R^m)^{\prime\prime}$ is a subalgebra of $\Gamma_{\tilde{Q}}(\R^{m+1})^{\prime\prime}$ which is the range of a normal faithful trace-preserving conditional expectation that intertwines the action of $m_{\phi}$, the sequence of norms on the right-hand side is non-decreasing, so
\begin{align*}
\norm{m_{\phi} : \Gamma_{T}(\h_{\R},U_t)^{\prime\prime} \rightarrow \Gamma_{T}(\h_{\R},U_t)^{\prime\prime}}_{cb} 
&\leqslant \lim_{m\rightarrow\infty} \norm{m_{\phi} :\Gamma_{\tilde{Q}}(\R^m)^{\prime\prime} \rightarrow \Gamma_{\tilde{Q}}(\R^m)^{\prime\prime}}_{cb}\\
&=\norm{m_{\phi} :\Gamma_{\tilde{Q}}(\ell^2)^{\prime\prime} \rightarrow\Gamma_{\tilde{Q}}(\ell^2)^{\prime\prime}}_{cb}. 
\end{align*}

\end{proof}

By  ~\cite[Theorem 5.5]{JZ15}, it is known that the mixed $q$-Gaussian von Neumann algebras have the complete metric approximation property. Hence, if we consider the map $F_n$ on  $\Gamma_{\tilde{Q}}(\ell^2)^{\prime\prime}$ and the commutative relation from Theorem \ref{intertwin}, it follows from ~\cite[Theorem 5.5 ]{JZ15} and ~\cite[Proposition 3.3]{A11}, that $\norm{F_n}_{cb}\leqslant C(q)n^2$.

\begin{thm}\label{Proj}
 Let $\Gamma_{T}(\h_{\R},U_t)^{\prime\prime} $ be a mixed $q$-deformed Araki–Woods factor. Let $F_n$ be the map defined on $\widetilde\Gamma_T(\h_{\R},U_t)$, defined by $F_nW(\xi) = \delta_{n,l}W(\xi),$ where $\xi\in\h^{\otimes l}.$ Then, $F_n$ extends to a completely bounded, normal map on $\Gamma_T(\h_{\R},U_t)^{\prime\prime}$ and $\norm{F_n}_{cb} \leqslant C(q)n^2.$
\end{thm}
\begin{proof}
Note that $F_n = m_{\phi_n}$ on $\widetilde\Gamma_T(\h_{\R},U_t)$, where $\phi_n$ is the Kronecker delta function $\phi_n(k) = \delta_{n,k}.$ By  Theorem \ref{intertwin},  we obtain 
 \begin{align*}
 \norm{F_n : \Gamma_{T}(\h_{\R},U_t)^{\prime\prime} \rightarrow \Gamma_{T}(\h_{\R},U_t)^{\prime\prime}}_{cb}& \leqslant\sup_n\norm{F_n : \Gamma_Q(\R^n)^{\prime\prime} \rightarrow \Gamma_Q(\R^n)^{\prime\prime}}_{cb}\\& =\norm{F_n : \Gamma_Q(\ell^2)^{\prime\prime} \rightarrow \Gamma_Q(\ell^2)^{\prime\prime}}_{cb}\\&\leqslant C(q)n^2.
\end{align*}
\end{proof}

\section{Complete metric approximation property}
In this section, we use the results obtained in previous sections to prove our main result. We quote the following lemma proved in \cite[Proposition 3.17]{HR11} for our main result.

 \begin{lem}\label{contraction}
There is a net of finite rank contractions $(T_k)_{k\in \Lambda}$ converging to the identity on $\h_{\C}$ pointwise, such that $T_k =\J T_k \J,$ for $k\in \Lambda$. 
 \end{lem}

\begin{thm}
The mixed $q$-deformed Araki-Woods factor $\Gamma_T(\h_{\R},U_t)^{\prime\prime}$ has the w*-complete metric approximation property.
\end{thm}
 \begin{proof}

 We define a net $\Gamma_{n,t,k}:= \Gamma(e^{-t}T_k)B_n$, where $n\in \N, t > 0, k \in \Lambda$, the finite-rank maps $T_k$ come from the Lemma \ref{contraction}, and $B_n = F_0 + \cdots + F_n = m_{\chi\set{0,1,\ldots,n}}$ is the radial multiplier which projects onto Wick words of length at most $n$. Each $\Gamma_{n,t,k}$ is a finite rank map on $\Gamma_T(\h_{\R}, U_t)^{\prime\prime}$; indeed, $B_n$ tells us that we have only Wick words of bounded length and $T_k$ tells us that we can only draw vectors from a finite-dimensional Hilbert space, so we are left with a space of the form $\oplus_{d=0}^n (\C^m)^{\otimes d}$, which is finite-dimensional. We will pass to a limit with $k\rightarrow\infty, n \rightarrow\infty$, and $t\rightarrow 0$. The rate of convergences of $t$ and $n$ will not be independent and will be chosen in a way that assures the convergence $\norm{\Gamma_{n,t,k}}$. Note that $\Gamma(e^{-t}I)F_k = e^{-kt}F_k$, where $I$ denotes the identity operator on $\h$.
         
 \begin{align*}
 \norm{\Gamma_{n,t,k}}_{cb}= \norm{\Gamma(e^{-t}T_k)B_n}_{cb}  &\leqslant \norm{\Gamma(e^{-t}I)B_n}_{cb}  \\
 & \leqslant \norm{\Gamma(e^{-t}I)(B_n-1+1)}_{cb}  \\
&\leqslant \norm{\Gamma(e^{-t}I)}_{cb} + \norm{\Gamma(e^{-t}I)(1- B_n)}_{cb} \\
& \leqslant 1+\displaystyle\sum_{k>n}e^{-kt}\norm{F_k}_{cb}\\
& \leqslant 1+C(q)'\displaystyle\sum_{k>n}e^{-kt}k^2,
\end{align*}
where $1$ denotes the identity on $\Gamma_T(\h_{\R},U_t)^{\prime\prime}$. The third inequality follows from Theorem \ref{second quantization}, and the last inequality follows from Theorem \ref{Proj}.
The series $\displaystyle\sum_{k\geq 0}e^{-kt}k^2$ is convergent for any $t>0$. This implies, $\displaystyle\sum_{k> n}e^{-kt}k^2$ converges to $0$ as $n\rightarrow \infty$. Therefore, we can choose the parameters $k, n \rightarrow\infty$ and $t\rightarrow 0$ such that the completely bounded norms of the operators $\Gamma_{n,t, k}$ tend to $1$. Then, the operators $\frac{\Gamma_{n,t,k}}{\norm{\Gamma_{n,t,k}}_{cb}}$ are completely contractive. Note that the denominators converge to $1$, and hence, the net is uniformly bounded. To check the ultraweak convergence of this net to $1$, it is enough to show their convergence on a dense set in strong operator convergence. This can be seen in the wick products of simple tensors. So, consider a wick product $W(\xi)$, where $\xi=\xi_{t_1}\otimes\cdots\otimes\xi_{t_l}$, for some $\xi_{t_1},\ldots,\xi_{t_l}\in\h_{\C}$ and $l\in\N$. Then, 
 \begin{align*}
  &\norm{\left( \frac{\Gamma_{n,t,k}}{\norm{\Gamma_{n,t,k}}_{cb}}(W(\xi))-W(\xi)\right)\Omega}\\ & =  \norm{\frac{1}{\norm{\Gamma_{n,t,k}}_{cb}}((\Gamma(e^{-t}T_k)B_n)(W(\xi))-W(\xi))\Omega}\\
  & = \norm{\frac{1}{\norm{\Gamma_{n,t,k}}_{cb}}(\Gamma(e^{-t}T_k)(\delta_{n,l}W(\xi))-W(\xi))\Omega}\\
  & = \norm{\frac{1}{\norm{\Gamma_{n,t,k}}_{cb}}(\delta_{n,l}W(e^{-t}T_k\xi_{t_1}\otimes\cdots\otimes e^{-t}T_k\xi_{t_l})\Omega-W(\xi_{t_1}\otimes\cdots\otimes \xi_{t_l})\Omega}\\
  & = \norm{\frac{1}{\norm{\Gamma_{n,t,k}}_{cb}}e^{-lt}(T_k\xi_{t_1}\otimes\cdots\otimes T_k\xi_{t_l})-(\xi_{t_1}\otimes\cdots\otimes \xi_{t_l})}\\
  & = \norm{\frac{1}{\norm{\Gamma_{n,t,k}}_{cb}}e^{-lt}(T_k\xi_{t_1}\otimes\cdots\otimes T_k\xi_{t_l} - \xi_{t_1}\otimes\cdots\otimes \xi_{t_l})}+\\
  & \qquad \norm{( \frac{1}{\norm{\Gamma_{n,t,k}}_{cb}}e^{-lt} -1) (\xi_{t_1}\otimes\cdots\otimes \xi_{t_l})}
\end{align*}
Note that $T_k$'s are finite rank contractions from Lemma \ref{contraction}, which converges to identity pointwise. Hence, taking limit $n,k\rightarrow\infty$, and $t\rightarrow 0$, using the fact that $\Omega$ is the cyclic separating vector for the standard representation of $\Gamma_T(\h_{\R},U_t)^{\prime\prime}$, we claim that $\frac{\Gamma_{n,t,k}}{\norm{\Gamma_{n,t,k}}_{cb}}$ converges to 1 in strong operator topology on the dense set of wick products. This completes the proof.

 \end{proof}
 
% \section{Declarations}
%\textbf{ Conflicts of interests:}  There is no conflict of interest.

\bibliographystyle{plain}
%\bibliography{ref.bib}

\end{document}